\documentclass[11pt]{article}%
\usepackage{amssymb,amsmath,amsfonts,amsthm,array,bm,bbm,color}%
\usepackage{graphicx}
\providecommand{\U}[1]{\protect\rule{.1in}{.1in}}

\numberwithin{equation}{section}

\newtheorem{theorem}{Theorem}[section]
\newtheorem{lemma}[theorem]{Lemma}
\newtheorem{corollary}[theorem]{Corollary}
\newtheorem{proposition}[theorem]{Proposition}

\newtheorem{definition}[theorem]{Definition}

\def\<{\langle}
\def\>{\rangle}
\def\d{{\rm d}}

\def\div{{\rm div}}
\def\E{\mathbb{E}}
\def\N{\mathbb{N}}
\def\P{\mathbb{P}}
\def\R{\mathbb{R}}
\def\T{\mathbb{T}}
\def\Z{\mathbb{Z}}

\def\eps{\varepsilon}

\begin{document}

\title{Stochastic inviscid Leray-$\alpha$ model with transport noise:\\ convergence rates and CLT}

\author{Dejun Luo\footnote{Email: luodj@amss.ac.cn}\quad Bin Tang\footnote{Email: tbdsj@amss.ac.cn} \bigskip \\
{\footnotesize Key Laboratory of RCSDS, Academy of Mathematics and Systems Science,}\\
{\footnotesize Chinese Academy of Sciences, Beijing 100190, China} \\
{\footnotesize School of Mathematical Sciences, University of Chinese Academy of Sciences, Beijing 100049, China}}

\maketitle

\vspace{-20pt}

\begin{abstract}
We consider the stochastic inviscid Leray-$\alpha$ model on the torus driven by transport noise. Under a suitable scaling of the noise, we prove that the weak solutions converge, in some negative Sobolev spaces, to the unique solution of the deterministic viscous Leray-$\alpha$ model. This implies that transport noise regularizes the inviscid Leray-$\alpha$ model so that it enjoys approximate weak uniqueness. Interpreting such limit result as a law of large numbers, we study the underlying central limit theorem and provide an explicit convergence rate.
\end{abstract}

\textbf{Keywords:} Leray-$\alpha$ model, transport noise, scaling limit, convergence rate, central limit theorem

\textbf{MSC (2020):} 60H15, 60H50

\section{Introduction}

We consider the inviscid Leray-$\alpha$ model (or the $\alpha$-approximation of Euler equations) on the torus $\T^d= \R^d/\Z^d\ (d=2,3)$, perturbed by multiplicative transport noise:
\begin{equation}\label{stoch-NS-eq}
  \left\{ \aligned
  & \d u +v\cdot \nabla u\,\d t + \d \nabla P = \nabla u \circ \d W , \\
  & u= v + \alpha^{2\gamma_0}(-\Delta)^{\gamma_0} v, \\
  & \nabla\cdot u= \nabla\cdot v=0.
  \endaligned \right.
\end{equation}
Here $u$ stands for the fluid velocity, and $P$ is the turbulent pressure whose precise form will be discussed in Section \ref{subs-precise-model} below; $\Delta$ is the Laplacian operator on $\T^d$, $\alpha> 0$ is a given length-scale, and $\gamma_0> 0$ is the regularizing parameter; $\circ\,\d$ means the Stratonovich stochastic differential and $W= W(t,x)$ is a space-time noise defined on some probability space $(\Omega,\mathcal F, \P)$, white in time, colored and divergence free in space, modelling transport-type perturbations of fluid small scales on larger ones. When $\alpha=0$, the above system reduces to the stochastic Euler equations with transport noise. It is easy to know that, for any $\gamma_0>0$, system \eqref{stoch-NS-eq} admits weak solutions $\{u_t\}_{t\in [0,T]}$ for all initial data $u_0\in H_\sigma$, the latter being the subspace of divergence free vector fields in $L^2(\T^d,\R^d)$ with zero mean; moreover, $\P$-a.s. $\|u_t \|_{L^2} \le \|u_0 \|_{L^2}$ for all $t\in [0,T]$. For fixed $\gamma_0 > \frac{d-2}4$ and suitably chosen noises $W=W^N(t,x),\, N\ge 1$, we will prove that weak solutions $u^N$ of \eqref{stoch-NS-eq} are close to the unique weak solution of the deterministic viscous Leray-$\alpha$ type model
  \begin{equation}\label{determ-NS-alpha}
  \left\{ \aligned
  & \partial_t \tilde u +\tilde v\cdot \nabla \tilde u + \nabla \tilde P = \kappa \Delta \tilde u, \\
  & \tilde u= \tilde v + \alpha^{2\gamma_0} (-\Delta)^{\gamma_0} \tilde v, \\
  & \nabla\cdot\tilde  u= \nabla\cdot \tilde v=0,
  \endaligned \right.
  \end{equation}
where $\kappa>0$ comes from the intensity of noise; we will provide explicit estimate on the distance $\|u^N -\tilde u\|_{C^0_t H^{-a}_x}$ for appropriate $a>0$, in terms of the parameters of noise $W^N(t,x)$. Moreover, interpreting the convergence of $u^N$ to the deterministic limit $\tilde u$ as a law of large numbers, we shall establish a corresponding central limit theorem for the Gaussian type fluctuation $U^N= (u^N -\tilde u)/\sqrt{\epsilon_N}$ with suitable $\epsilon_N\to 0$, identifying the equation fulfilled by the limit $U$ and studying the strong convergence rate of $U^N$ to $U$. Before proceeding to more detailed statements of our main results, we briefly recall some previous works in the literature.

We start from the deterministic setting. To show the existence of solutions to the classical Navier-Stokes equations, i.e. $\gamma_0=0$ in \eqref{determ-NS-alpha}, Leray considered in the seminal work \cite{Leray33} the following approximating equations:
  \begin{equation*}\label{determ-NS-alpha-1}
  \left\{ \aligned
  & \partial_t u + v^\alpha \cdot \nabla u + \nabla P = \kappa \Delta u, \\
  & v^\alpha = \phi_\alpha\ast u, \\
  & \nabla\cdot u= 0,
  \endaligned \right.
  \end{equation*}
where $\phi_\alpha$ are smoothing kernels. Inspired by the Lagrangian averaged Navier-Stokes-$\alpha$ (LANS-$\alpha$) model of turbulence (also known as the viscous Camassa-Holm equations, see e.g. \cite{CFHOTW98, FHT02}), Cheskidov et al. \cite{CHOT05} adopted the smoothing operator associated with the Green function of the Helmholtz operator:
  $$v^\alpha= (1-\alpha^2\Delta )^{-1} u, $$
and they called the new system as the Leray-$\alpha$ model; here, $\alpha$ is considered as a given length scale. It turns out that, for a wide range of Reynolds numbers, this model compares successfully with empirical data from turbulent channel and pipe flows. In the subsequent paper \cite{OT07}, Olson and Titi considered a family of equations which interpolates between the Navier-Stokes equations and the LANS-$\alpha$ model:
  $$  \left\{ \aligned
  & \partial_t u - v^\alpha \times(\nabla\times u) + \nabla P = -\nu\, (-\Delta u)^{\gamma_1} u + f, \\
  & v^\alpha = \big(1+ (-\alpha^2 \Delta)^{\gamma_0} \big)^{-1} u, \\
  & \nabla\cdot u= \nabla\cdot v^\alpha=0.
  \endaligned \right. $$
Note that if $\gamma_1=1$ and $\gamma_0=0$, then the above system reduces to the Navier-Stokes equations, while if $\gamma_1= \gamma_0=1$, it is the LANS-$\alpha$ model. They proved that if $\nu>0$ and $2\gamma_0 + 4\gamma_1 \ge 5$, or if $\nu=0$ and $2\gamma_0 \ge 5$, then the above system admits unique global smooth solutions depending continuously on the initial data. We refer to the introductions of \cite{CHOT05, OT07} for more discussions on the background and related models for turbulence, and \cite{Tao09, Yam12, BMR14, BF17} for studies on more general models. In particular, a special case of \cite[Theorem 1.2]{BMR14} shows that if $\gamma_0\ge (d-2)/4$, then \eqref{determ-NS-alpha} admits a global smooth solution for every smooth initial condition.

There have also been numerous studies of Leray-$\alpha$ type models in the stochastic setting. An early work by Deugoue and Sango \cite{DS09} established the existence of probabilistic weak solutions to the stochastic 3D Navier-Stokes-$\alpha$ model and discussed also the uniqueness. Chueshov and Millet \cite{CM10} studied an abstract stochastic evolution equation of the form
  $$\d u + (Au + B(u,u) + R(u))\,\d t = \sqrt{\eps} \, \sigma(t,u)\,\d W_t $$
in a Hilbert space; the conditions on the linear operator $A$, the bilinear mapping $B$ and the operator $R$ cover many 2D fluid dynamical models, as well as the 3D Leray-$\alpha$ model and some shell models of turbulence. We refer to \cite{BHR15, FHR16} for results on the well-posedness, irreducibility and exponential mixing of stochastic 3D Leray-$\alpha$ model driven by pure jump noise, to \cite{BR16} for $\alpha$-approximation of stochastic Leray-$\alpha$ model to the stochastic Navier-Stokes equations, to \cite{HLL21} for asymptotic log-Harnack inequality and ergodicity of 3D Leray-$\alpha$ model with degenerate type noise. We also mention that Barbato et al. \cite{BBF14} considered the stochastic inviscid Leray-$\alpha$ model \eqref{stoch-NS-eq} with $\gamma_0=1$, and proved that suitable transport noise restores uniqueness in law of weak solutions.

Stochastic fluid dynamical equations with transport noise have been studied for a long time, see e.g. \cite{BCF92, MikRoz04, MikRoz05} for some early works; recently there are some rigorous justifications of such noise in fluid equations, cf. \cite{Holm15} for arguments based on variational principles, \cite{FlaPap21, FlaPap22} for results in 2D case and also \cite{DP22} where the examples include 3D Navier-Stokes equations and primitive equations. In several papers \cite{BFM16, FL19, BM19, LC22}, well-posedness of the stochastic 2D Euler equations in vorticity form and driven by transport noise was established for initial data of various regularity; see \cite{GoodCris} for studies on the Navier-Stokes equations with stochastic Lie transport. In recent years, there have been intensive investigations \cite{Galeati20, FGL21a, Luo21a} on scaling limits of SPDEs with transport noise to deterministic limit equations with an additional viscous term, cf. the earlier work \cite{FL20} where the limit equation is driven by additive space-time white noise. The enhanced dissipation was used in \cite{FL21} to suppress possible blow-up of vorticity for the 3D Navier-Stokes equations, see also \cite{FGL21b} for related results on more general nonlinear equations and \cite{Agresti22} on systems of reaction-diffusion equations. The weak convergence results of \cite{Galeati20, FGL21a, Luo21a} have been improved in \cite{FGL21c} by establishing quantitative convergence rates; the method makes use of mild formulations of both approximating and limit equations, and it works as well for producing explicit estimates on blow-up probabilities of solutions to various stochastic equations with transport noise, see \cite[Theorem 1.5]{FGL21c} for the 2D Keller-Segel model and \cite[Theorem 2.4]{Luo21b} for the vorticity form of 3D Navier-Stokes equations. There are also results on dissipation enhancement by  transport noise for heat equations in bounded domains \cite{FGL21d} or infinite channels \cite{FlaLuongo22}. In the recent paper \cite{GL22}, we studied the large deviation principle (LDP) and central limit theorem (CLT) associated with the above scaling limit results; in particular, using the approach of mild formulations, we were able to find the strong convergence rate for the CLT. Similar results were obtained in \cite{LW22} for a class of stochastic dyadic models, where the noise is allowed to transfer energy among distant components. For related studies on other models with transport noise, we refer to \cite{Lange} for an averaged version of the Navier-Stokes equations which was motivated by \cite{Tao16}, and \cite{CarLuon} for a two-layer quasi-geostrophic model.

The purpose of this work is to establish similar results for the stochastic inviscid Leray-$\alpha$ model \eqref{stoch-NS-eq} with transport noise. In particular, as a direct consequence of Theorem \ref{thm-main} below, one can easily prove an analogue of \cite[Corollary 1.4]{FGL21c}, namely, the Wasserstein distance between probability laws of weak solutions to \eqref{stoch-NS-eq} vanishes in the limit. This result implies that transport noise regularizes asymptotically the Leray-$\alpha$ model so that it enjoys ``approximate weak uniqueness'' of solutions; see \cite[Section 6.1]{FGL21a} for related discussions.

In the rest of this section, we describe the precise equations we are going to study in the paper and give more precise statements of the main results. The exact value of $\alpha>0$ is not important and we will fix it to be 1 in the sequel.

\subsection{The precise model} \label{subs-precise-model}

As in \cite{Galeati20, FL21}, the space-time noise used to perturb the equations has the form
  \begin{equation}\label{noise-1}
  W(t, x)= \sqrt{C_{d} \kappa}\, \sum_{k \in \Z_0^d} \sum_{i=1}^{d-1} \theta_{k} \sigma_{k, i}(x) W^{k, i}_{t} ,
  \end{equation}
where $C_d=d/(d-1)$ is a normalizing constant, $\kappa>0$ is the noise intensity, and $\Z_0^d$ is the set of nonzero lattice points; $\{\theta_k\}$ are coefficients of noise, $\{\sigma_{k,i}\}$ are divergence free vector fields on $\T^d$ and $\{W^{k,i}\}$ are independent complex Brownian motions, see Section \ref{subs-noise} for more details. We shall write $\sum_{k \in \Z_0^d} \sum_{i=1}^{d-1}$ simply as $\sum_{k,i}$ in the sequel.

Inspired by discussions in \cite[Section 2.1]{MikRoz04}, we set the pressure term in \eqref{stoch-NS-eq} as follows:
\begin{equation*}
  \d \nabla P= \nabla P_1\,\d t + \sum_{k,i} \nabla P_2^{k,i} \circ \d W^{k, i}_{t};
\end{equation*}
as mentioned at the beginning, we take $\alpha=1$, and thus SPDE \eqref{stoch-NS-eq} can be rewritten as
\begin{equation}\label{stoch-NS-eq2}
  \left\{ \aligned
  & \d u +(v\cdot \nabla u + \nabla P_1)\,\d t = \sqrt{C_{d} \kappa}\, \sum_{k,i} \theta_{k} \sigma_{k, i} \cdot \nabla u \circ \d W^{k,i}_t - \sum_{k,i} \nabla P^{k,i}_2 \circ \d W^{k,i}_t , \\
  & u= v + (-\Delta)^{\gamma_0} v, \\
  & \nabla\cdot u= \nabla\cdot v=0.
  \endaligned \right.
\end{equation}
We see that the pressures $P_1,P_2$ should satisfy the following equations:
\begin{equation}\label{eq-pressures}
  \left\{ \aligned
    & \nabla\cdot (v\cdot \nabla u) + \Delta P_1 = 0 ,\\
    & \sqrt{C_{d} \kappa} \, \theta_{k} \nabla \cdot \left( \sigma_{k,i} \cdot \nabla u \right) - \Delta P_2^{k,i}= 0, \quad \forall\, k \in \Z_0^{d}, \,  i=1,2,\dots ,d-1.
  \endaligned \right.
\end{equation}
Applying the Helmholtz-Leray projector $\Pi$ to the first equation of \eqref{stoch-NS-eq2} yields
  $$ \d u + \Pi(v\cdot \nabla u)\,\d t =\sqrt{C_{d} \kappa} \, \sum_{k,i} \theta_{k} \Pi(\sigma_{k,i} \cdot \nabla u) \circ \d W^{k,i}_t; $$
it has the equivalent It\^o form:
\begin{equation*}
  \d u + \Pi(v\cdot \nabla u)\,\d t =\sqrt{C_{d} \kappa} \, \sum_{k,i} \theta_{k} \Pi(\sigma_{k,i} \cdot \nabla u) \, \d W^{k,i}_t + S^{(d)}_{\theta}(u) \, \d t,
\end{equation*}
where the Stratonovich-It\^o corrector $S_\theta^{(d)}(\cdot)$ is defined as
\begin{equation} \label{def-Stratonovich-Ito-correctors}
  S_\theta^{(d)}(u) = C_d \kappa \sum_{k,i} \theta_{k}^2 \Pi \left[ \sigma_{k,i} \cdot \nabla \Pi \left( \sigma_{-k,i} \cdot \nabla u \right)  \right].
\end{equation}
So the equations considered in this paper can be written as
  \begin{equation}\label{stoch-NS-eq3}
  \left\{ \aligned
  & \d u + \Pi(v\cdot \nabla u)\,\d t =\sqrt{C_{d} \kappa} \, \sum_{k,i} \theta_{k} \Pi(\sigma_{k,i} \cdot \nabla u) \, \d W^{k,i}_t + S^{(d)}_{\theta}(u) \, \d t, \\
  & u= v + (-\Delta)^{\gamma_0} v.
  \endaligned \right.
  \end{equation}
For any $\gamma_0>0$ and $u_0\in H_\sigma$ (see Section \ref{sec-notations-noise} for its definition), the existence of (probabilistically and analytically) weak solutions to \eqref{stoch-NS-eq3} will be briefly proved in Appendix \ref{sec-appendix-A}; once we have a solution $u$ to this system, it is possible to recover the pressure terms $P_1$ and $P_2^{k,i}$ through the equations \eqref{eq-pressures}.
If $\gamma_0\geq\frac{d+2}{4}$, the deterministic inviscid Leray-$\alpha$ model is already globally well posed for smooth initial data, hence we always assume $\gamma_0 < \frac{d+2}{4}$ in \eqref{stoch-NS-eq3}.

In the sequel, we will often write only the first equation in \eqref{stoch-NS-eq3}, and keep in mind that $v=K(u)$, where the linear operator $K:H_{\sigma} \mapsto H_{\sigma}$ is defined as
$$ K(u):=(1+ (-\Delta)^{\gamma_0})^{-1} u=\sum_{k \in \Z_0^d} \frac{\hat{u}_k}{1+|2 \pi k|^{2 \gamma_0}} e_k,$$
where $\{\hat{u}_k\}_{k \in \Z_0^d}$ are the Fourier coefficients of $u$. Then for any $s \in \R $, we have
$$ \|K(u)\|_{H^s}^2 =\sum_{k \in \Z_0^d} |k|^{2s} \Big(\frac{ |\hat{u}_k|}{1+|2 \pi k|^{2 \gamma_0}} \Big)^2 \leq \sum_{k \in \Z_0^d} |k|^{2(s-2\gamma_0)} |\hat{u}_k|^2 = \|u\|_{H^{s-2\gamma_0}}^2 .$$

\subsection{Main results}

The main results proved in the paper consist of (i) quantitative convergence rates of stochastic Leray-$\alpha$ model of Euler equations to the deterministic Leray-$\alpha$ model for Navier-Stokes equations, and (ii) a CLT type result concerning the Gaussian type fluctuations underlying the scaling limit in (i).

To this end, we choose in \eqref{noise-1} a special sequence of noise coefficients $\{\theta^N \}_N \subset \ell^2(\Z^d_0)$:
\begin{equation} \label{theta-N-def}
  \theta^N_k = \sqrt{\epsilon_N} \frac{1}{|k|^\gamma} {\bf 1}_{\{1  \leq |k|\leq N\}}, \quad k\in \Z^d_0 ,
\end{equation}
where $0< \gamma < \frac{d}{2}$ and
\begin{equation}\label{def-epsilon}
  \epsilon_N := \bigg(\sum_{1 \leq |k| \leq N} \frac{1}{|k|^{2\gamma}} \bigg)^{-1} \sim  N^{2 \gamma-d}.
\end{equation}
Note that this choice is different from those in \cite{FL21, Luo21b}, where we assumed, for any fixed $\gamma>0$,
  $$\theta^N_k = \sqrt{\epsilon_N} \frac{1}{|k|^\gamma} {\bf 1}_{\{N  \leq |k|\leq 2N\}}, \quad k\in \Z^d_0 $$
for a corresponding $\epsilon_N$. We stress that the choice \eqref{theta-N-def} is necessary for proving CLT type result (clearly the upper bound $N$ is not essential); however, allowing lower modes of noise makes it more difficult to show the convergence rates of Stratonovich-It\^o correctors $S^{(d)}_{\theta^N}(\cdot)$, see the proof of Theorem \ref{thm-Ito-corrector} in  Appendix \ref{sec-appendix-B}.

Let $u^N$ be a weak solution to \eqref{stoch-NS-eq3} with $\theta= \theta^N$, namely,
  \begin{equation} \label{eq-inviscid-Leray}
  \d u^N + \Pi(v^N\cdot \nabla u^N)\,\d t =\sqrt{C_{d} \kappa} \, \sum_{k,i} \theta^N_{k} \Pi(\sigma_{k,i} \cdot \nabla u^N) \, \d W^{k,i}_t + S^{(d)}_{\theta^N}(u^N) \, \d t,
  \end{equation}
where $v^N = K(u^N)$ and $u^N_0=u_0 \in H_\sigma$. Thanks to Theorem \ref{thm-Ito-corrector}, we expect that the solutions $\{u^N\}_{N\ge 1}$ are close to the solution of the deterministic viscous Leray-$\alpha$ model:
\begin{equation} \label{regular-limit-eq}
  \partial_t \tilde u + \Pi(\tilde v\cdot \nabla\tilde u )= C_d^{\prime} \kappa\Delta \tilde u
\end{equation}
with the same initial condition $\tilde u_0 =u_0 \in H_\sigma$; here $\tilde v= K(\tilde u)$ and
  \begin{equation} \label{C-d-prime}
  C_d^{\prime}= \begin{cases}\frac{1}{4}, & d=2, \\
  \frac{3}{5}, & d=3.
  \end{cases}
  \end{equation}
This strange constant is due to the Helmholtz-Leray projection $\Pi$ in the corrector $S^{(d)}_{\theta^N}(\cdot)$, see Theorem \ref{thm-Ito-corrector}; if there were no $\Pi$, then $C'_d$ would be 1 in both cases. As mentioned above, for $\gamma_0\ge \frac{d-2}4$, equation \eqref{regular-limit-eq} admits a unique solution $\tilde u$ satisfying the energy estimate
  \begin{equation} \label{energy-ineq}
  \|\tilde u_t \|_{L^2}^2 + \kappa \int_0^t \|\nabla \tilde u_s \|_{L^2}^2\,\d s \lesssim \|u_0 \|_{L^2}^2 \quad \mbox{for all } t>0.
  \end{equation}
We remark that, although \eqref{regular-limit-eq} is well posed also for $\gamma_0= \frac{d-2}4$, due to technical reasons, we will work under the condition $\gamma_0> \frac{d-2}4$, see Theorems \ref{thm-main} and \ref{thm:CLT} below.

Now we are ready to state the first main result which will be proved in Section \ref{sec-proof-converg-rates}.

\begin{theorem}[Quantitative convergence rates]\label{thm-main}
Fix $\gamma_0 \in (\frac{d-2}{4},\frac{d+2}{4})$, $\alpha\in (0,(2\gamma_0) \wedge 1)$ and $q>\max\{2,\frac{4}{4\gamma_0-d+2}\}$. Then for any $T<\infty$ there exists a constant $C=C(T,q,\gamma_0)>0$ such that for any $\delta \in (0,\alpha)$ it holds
  $$\E \bigg[\sup_{t\leq T} \|u^N_t -\tilde u_t\|_{H^{-\alpha}}^q \bigg] \lesssim \kappa^{q\delta/2} \epsilon_N^{q(\alpha-\delta)/d} \|u_0\|_{L^2}^q \exp\bigg[ C \frac{\|u_0 \|_{L^2}^q}{\kappa^{q}} (\kappa^{q (\gamma_0-\frac{d-2}{4})} T+1) \bigg]. $$
Furthermore, if $\gamma >\frac{d-2}{2}$, we have a slightly better estimate:
$$\E \bigg[\sup_{t\leq T} \|u^N_t -\tilde u_t\|_{H^{-\alpha}}^q \bigg] \lesssim \kappa^{q\delta/2} \epsilon_N^{q(\alpha-\delta)/2} \|u_0\|_{L^2}^q \exp\bigg[ C \frac{\|u_0 \|_{L^2}^q}{\kappa^{q}} (\kappa^{q (\gamma_0-\frac{d-2}{4})} T+1) \bigg]. $$
\end{theorem}

Note that the second estimate improves the first one in the 3D case since the denominator $d$ in the exponent of $\epsilon_N$ is replaced by 2. We point out that we have assumed for simplicity that $u^N_0 =u_0$ for all $N\ge 1$; in general, if $\sup_{N\geq 1}\|u^N_0\|_{L^2} \leq R$ and $u^N_0 \in H_{\sigma}$ converges weakly to some $u_0$, then the right-hand side of the first estimate should be replaced by
  $$\Big(\|u^N_0- u_0\|_{H^{-\alpha}}^q + \kappa^{q\delta/2} \epsilon_N^{q(\alpha-\delta)/d} R^q \Big) \exp\bigg[ C \frac{R^q}{\kappa^{q}} (\kappa^{q (\gamma_0-\frac{d-2}{4})} T+1) \bigg]. $$
Similar change applies to the second estimate.

Next we turn to study the Gaussian type fluctuation
  $$U^N:= \frac{u^N- \tilde u}{\sqrt{\epsilon_N}}, \quad N\ge 1;$$
one easily sees that $U^N$ satisfies the following equation in a weak sense:
  \begin{equation*}
  \aligned
  \d U^N + \Pi \left( V^N \cdot \nabla u^N + \tilde{v} \cdot \nabla U^N \right) \, \d t
  &= C_d^{\prime}\kappa \Delta U^N \, \d t + \frac{1}{\sqrt{\epsilon_N}} \bigl(  S_{\theta^N}^{(d)} (u^N)-C_d^{\prime}\kappa \Delta u^N \bigr) \, \d t \\
   &\quad + \sqrt{C_d \kappa} \sum_{|k| \leq N} \sum_{i=1}^{d-1} \frac{1}{|k|^{\gamma}} \Pi \bigl( \sigma_{k,i} \cdot \nabla u^N \bigr)\, \d W_{t}^{k,i},
  \endaligned
  \end{equation*}
where $V^N = K(U^N)= (v^N-\tilde{v} ) / \sqrt{\epsilon_N}$. Letting $N \to \infty$, any limit point $U$ of $U^N$ is expected to solve
  \begin{equation*}
  \d U  + \Pi \left( V \cdot \nabla \tilde{u} + \tilde{v} \cdot \nabla U \right) \d t =C_d^{\prime} \kappa \Delta U \, \d t  + \sqrt{C_d \kappa} \sum_{k,i} \frac{1}{|k|^{\gamma}} \Pi(\sigma_{k,i}\cdot \nabla \tilde{u} )\, \d W_{t}^{k,i}
  \end{equation*}
with $U_0=0$; here $V = K(U)$ and $\tilde v= K(\tilde u)$. Recall that $\tilde u$ is the unique solution to \eqref{regular-limit-eq} with $\tilde u_0 = u_0$. It is not difficult to establish the existence and uniqueness of probabilistically strong solutions to the limit equation, see Corollary \ref{cor:CLT-limit-eq} below. This fact is important since it enables us to define the solution $U$ on the same probability space as for $U^N$, the latter being only a weak solution. As a consequence, we can estimate the convergence rate of $U^N$ to $U$ in suitable negative Sobolev norm, by following some ideas in \cite[Section 3.2]{GL22}. It turns out to be quite complicated to treat 2D and 3D cases at the same time, due to the presence of several parameters like dimension $d$, regularizing parameter $\gamma_0$, and also $\gamma$ in \eqref{theta-N-def}. We are mainly interested in the 3D case and want to keep $\gamma_0$ as small as possible, therefore, we restrict $\gamma$ in a slightly special range.

\begin{theorem}[Central limit theorem]\label{thm:CLT}
Let $d=3$, $u_0 \in H_\sigma$ be given, $N \in\N$ and $u^N$ be a weak solution to \eqref{eq-inviscid-Leray}; define $U^N,\,U$ as above. We assume
  $$ \frac{1}{4}<\gamma_0 < \frac{5}{4} \quad \mbox{and}\quad 1<\gamma<\frac{3}{2}. $$
For any $\alpha_0 \in \big(\frac{1}{2}, 1 \wedge (2\gamma_0) \big)$, $q > \max \{ 2, \frac{4}{4\gamma_0-1} \}$ and $\eps>0$ small enough, it holds
  \begin{equation} \label{eq:rate-CLT-3D}
    \sup_{t\in [0,T]} \E \big[\| U^N_t-U_t\|_{H^{-\alpha_0}}^q \big] \lesssim N^{-q \frac{3-2\gamma}{2}(\alpha_0- \frac12)+\eps}.
  \end{equation}
\end{theorem}

We finish the introduction with the structure of the paper. In section \ref{sec-notations-noise} we introduce some notation used frequently in the sequel, and give the precise definitions of $\sigma_{k,i}$ and $W^{k,i}$ in the noise \eqref{noise-1}. Section 3 contains some preliminary results, including elementary estimates on the transport term (Section \ref{subs-classical-estim}), convergence rates of Stratonovich-It\^o correctors $S^{(d)}_{\theta^N}$ (Section \ref{subs-correctors}) and some estimates on stochastic convolutions in the last subsection. Theorems \ref{thm-main} and \ref{thm:CLT} will be proved in Sections \ref{sec-proof-converg-rates} and \ref{sec-CLT}, respectively. Finally, we present in Appendix \ref{sec-appendix-A} a brief proof of the existence of weak solutions to \eqref{stoch-NS-eq3}, and treat in Appendix \ref{sec-appendix-B} the convergence rates of the correctors $S^{(d)}_{\theta^N}$, proving Theorem \ref{thm-Ito-corrector}.

\section{Functional setting and choice of noise} \label{sec-notations-noise}

In this section, we fix some notation used in the paper and give the precise choice of noises. Let $\T^d=\R^d/\Z^d$ be the $d$-dimensional torus; $\Z_0^d = \Z^d \setminus \{0\}$ is the set of nonzero lattice points.
The notation $a_N \sim b_N$ means that the limit $\lim_{N \rightarrow +\infty} \frac{a_N}{b_N}=C>0$.
By $a \lesssim b$ we mean that there exists a constant $C>0$ such that $a \leq C b$; if we want to emphasize the dependence of $C$ on some parameters, e.g. $\delta,T$, then we write $a \lesssim_{\delta,T} b $.

As we assume the noise is spatially divergence free, it is clear that the spatial average of solutions to \eqref{stoch-NS-eq} is preserved; hence, we shall assume for simplicity that the function spaces in this paper consist of functions on $\T^d$ with zero average.

Let $\{e_k \}_k$ be the usual complex basis of $L^2(\T^d,  \mathbb C)$. Let $\mathcal{H}$ be the space of formal Fourier series $\mathcal{H}=\big\{ u=\sum_{k \in \Z_0^d} \hat{u}_k e_k: \hat{u}_k \in \mathbb{C}^d \big\}$ and $\mathcal{H}_{\sigma}$ be the subspace of $\mathcal{H}$ consisting of divergence free vector fields $ \mathcal{H}_{\sigma}=\left\{ u \in \mathcal{H} : k \cdot \hat{u}_k =0, \, \forall k \in \Z_0^d \right\} $. The Helmholtz-Leray projection operator $\Pi$ maps $\mathcal{H}$ to $\mathcal{H}_{\sigma}$:
  $$ \Pi u := \Pi \sum_{k \in \Z_0^d} \hat{u}_k e_k = \sum_{k \in \Z_0^d} \Big(\hat{u}_k-\frac{k \cdot \hat{u}_k}{|k|^2}k \Big) e_k .$$ 
For $s \in \R$, we define Sobolev space $H^{s}(\T^d;\R^d)$ as
$$  H^{s}(\T^d;\R^d)=\left\{ u \in \mathcal{H}: \| u \|_{H^s} < \infty, \, \overline{\hat{u}_k} =\hat{u}_{-k}\right\},$$
where $\|u\|_{H^s}^2= \sum_{k \in \Z_0^d} |2 \pi k|^{2s} |\hat{u}_k|^2 $, and we define $H^{s}_{\sigma}(\T^d;\R^d)=H^{s}(\T^d;\R^d) \cap \mathcal{H}_{\sigma} $. For simplicity, we denote $H^{s}(\T^d;\R^d)$ and $H^{s}_{\sigma}(\T^d;\R^d)$ by $H^s$ and $H^s_{\sigma}$, respectively.
We denote the spaces $H^{0}$ and $H^{0}_{\sigma}$ by $H$ and $H_{\sigma}$, respectively, with the same norm $\|u\|_{L^2}=\|u\|_{H^0}$. We identify the continuous dual space of $H^{s}$ as $H^{-s}$ with the pairing given by $\langle u, v\rangle=\sum_{k \in \Z_0^d}\left(\hat{u}_{k} \cdot \hat{v}_{-k}\right)$.
Note that $H^{s+\varepsilon}$ is compactly embedded in $H^{s}$ for any $\varepsilon>0$. Similarly, $H_{\sigma}^{s+\varepsilon}$ is compactly embedded in $H_{\sigma}^{s}$. For simplicity, we denote the space $C([0,T];H^s(\T^d;\R^d))$ by $C^0_t H^{s}_x$ with norm $\|u\|_{C^0 H^s}= \sup_{t \in [0,T]} \|u(t)\|_{H^s}$.

For $1 \leq p < +\infty $, we define $\ell^p (\Z_0^d)$ as the space of $p$-order summable sequences indexed by $\Z_0^d$ with the norm $ \|a\|_{\ell^p}: = \big(\sum_{k \in \Z_0^d} |a_k|^p \big)^{1/p}$; $\ell^{\infty} (\Z_0^d)$ is the space of bounded sequences with the norm $ \|a\|_{\ell^{\infty}}: = \sup_{k \in \Z_0^d} |a_k| $.

\subsection{Choice of noise}\label{subs-noise}

As mentioned in Section \ref{subs-precise-model}, the space-time noise used to perturb \eqref{stoch-NS-eq} takes the form
\begin{equation}\label{noise}
    W(t, x)= \sqrt{C_{d} \kappa}\, \sum_{k \in \Z_0^d} \sum_{i=1}^{d-1} \theta_{k} \sigma_{k, i}(x) W^{k, i}_{t} ,
\end{equation}
where $C_{d}=d/(d-1)$ is a normalizing constant,  $\kappa>0$ is the noise intensity and $\theta \in \ell^{2}(\Z^d_0)$. $\{W^{k, i}:k\in\mathbb{Z}^{d}_{0},  i=1, \ldots, d-1\}$
are standard complex Brownian motions defined on a filtered probability space $(\Omega,  \mathcal F,  (\mathcal F_t),  \P)$, satisfying
\begin{equation}\label{noise.1}
    \overline{W^{k, i}} = W^{-k, i},  \quad\big[W^{k, i}, W^{l, j} \big]_{t}= 2t \delta_{k, -l} \delta_{i, j} .
\end{equation}
$\{\sigma_{k, i}: k\in\mathbb{Z}^{d}_{0},  i=1, \ldots, d-1\}$ are divergence free vector fields on $\T^d$ defined as
\begin{equation}\label{sigma}
    \sigma_{k, i}(x) = a_{k, i} e_{k}(x) ,
\end{equation}
where $\{a_{k, i}\}_{k, i}$ is a subset of the unit sphere $\mathbb{S}^{d-1}$ such that: (i) $a_{k, i}=a_{-k, i}$ for all $k\in \mathbb{Z}^{d}_{0}, \,  i=1, \ldots, d-1$; (ii) for fixed $k$,  $\{a_{k, i}\}_{i=1}^{d-1}$ is an ONB of $k^{\perp}=\{y\in\mathbb{R}^{d}:y\cdot k=0 \}$.
It holds that $\sigma_{k, i}\cdot \nabla e_k = \sigma_{k, i}\cdot \nabla e_{-k} \equiv 0$ for all $k\in \Z^d_0$ and $1\leq i\leq d-1$.
We shall always assume that $\theta $ is symmetric, i.e. $\theta_k = \theta_l$ for all $k, l\in \Z^d_0$ satisfying $|k|=|l|;$ and $\|\theta \|_{\ell^2} =1$.

In the main results of this article, we consider the special sequence of coefficients $\theta^N$ defined in \eqref{theta-N-def}; nevertheless, we shall provide in Appendix \ref{sec-appendix-A} a brief proof of existence of weak solutions to \eqref{stoch-NS-eq} with the above general noise.

\section{Some preparations}

This section contains some lemmas and estimates that will be used frequently in the sequel.

\subsection{Some classical estimates} \label{subs-classical-estim}

The following two lemmas will be useful when dealing with convolution sums. One can jump to Lemma \ref{cor-convolution sum} if one is not interested in the technical proof of Lemma \ref{lem-convolution sum}. We say that a sequence $A=\{ A_k \}_{k \in \Z_0^d} \in \R^{\Z_0^d}$ is  non-negative if $A_k \ge 0$ for all $k \in \Z_0^d$; we make the convention that $A_0 =0$, a condition fulfilled in applications below.

\begin{lemma} \label{lem-convolution sum}
Given $s \geq 0$ and $p_0 \geq 1$, assume $p_1\in [1, +\infty]$ and $p_2 \in \big[1, \frac{p_0}{p_0-1} \big]$ satisfy
  \begin{equation} \label{condition-convolution sum}
  \begin{cases}
  0 \leq  1-\frac{1}{p_0 p_1}-\frac{1}{p_2}< \frac{s}{d}, & \mbox{if } s>0 ; \\
  \frac{1}{p_0 p_1}+\frac{1}{p_2}=1, &  \mbox{if } s=0.
  \end{cases}
  \end{equation}
Let $A=\{ A_k \}_{k \in \Z_0^d} $, $B=\{ B_k \}_{k \in \Z_0^d} $ and $C=\{ C_k \}_{k \in \Z_0^d} $ be non-negative sequences; assume the new sequences  $ A^{(s)}= \{ |k|^{p_0 s} A_k \}_{k\in \Z_0^d} \in \ell^{p_1}(\Z_0^d)$, $ B^{(s)}= \{ |k|^{s} B_k \}_{k\in \Z_0^d} \in \ell^{p_2}(\Z_0^d)$ and $ \ C^{(s)}= \{ |k|^{-s} C_k \}_{k\in \Z_0^d} \in \ell^{p_0}(\Z_0^d)$. Then  we have
  $$ \sum_{k \in \Z_0^d} A_{k} \Big( \sum_{j \in \Z_0^d} B_{j} C_{k-j} \Big)^{p_0} \lesssim \big\|A^{(s)} \big\|_{\ell^{p_1}} \big\|B^{(s)} \big\|_{\ell^{p_2}}^{p_0} \big\|C^{(s)} \big\|_{\ell^{p_0}}^{p_0} .$$
\end{lemma}

\begin{proof}
First, for $1 \leq p \leq +\infty$ and $1 \leq \tilde{p}=\frac{p_0 p}{p_0 p-1} \leq \frac{p_0}{p_0-1}$, we will prove the following preliminary estimate: for any non-negative sequences $a=\{ a_k \}_{k \in \Z_0^d} \in \ell^{p}(\Z_0^d) $, $b=\{ b_k \}_{k \in \Z_0^d} \in \ell^{\tilde{p}}(\Z_0^d) $ and $c=\{ c_k \}_{k \in \Z_0^d} \in \ell^{p_0}(\Z_0^d)$, it holds
  \begin{equation} \label{eq-convolution sum}
    \sum_{k \in \Z_0^d} a_{k} \Big( \sum_{j \in \Z_0^d} b_{j} c_{k-j} \Big)^{p_0} \leq \|a\|_{\ell^{p}} \|b\|_{\ell^{\tilde{p}}}^{p_0} \|c\|_{\ell^{p_0}}^{p_0} .
  \end{equation}
Indeed, let $p'$ be the conjugate number of $p$, then by H\"older's inequality, we have
  $$  \sum_{k \in \Z_0^d} a_{k} \Big( \sum_{j \in \Z_0^d} b_{j} c_{k-j} \Big)^{p_0} \leq \|a\|_{\ell^{p}} \bigg[ \sum_{k \in \Z_0^d} \Big( \sum_{j \in \Z_0^d} b_{j} c_{k-j} \Big)^{p_0 p'} \bigg]^{\frac{1}{p'}} = \|a\|_{\ell^{p}}  \| b \ast c \|_{\ell^{p_0 p'}}^{p_0} .$$
  By the relation between $p_0$, $p$ and $\tilde{p}$, we have
  $$ \frac{1}{p_0 p'} = \frac{1}{p_0}-\frac{1}{p_0 p}=\frac{1}{p_0} + \frac{1}{\tilde{p}}-1 .$$
Applying Young's inequality for convolutions gives us
  $$ \sum_{k \in \Z_0^d} a_{k} \Big( \sum_{j \in \Z_0^d} b_{j} c_{k-j} \Big)^{p_0} \leq \|a\|_{\ell^{p}}  \| b \ast c \|_{\ell^{p_0 p'}}^{p_0} \leq \|a\|_{\ell^{p}} \|b \|_{\ell^{\tilde{p}}}^{p_0} \|c \|_{\ell^{p_0}}^{p_0} .$$

The above inequality already implies the case $s=0$, so we assume $s>0$ in the sequel. By the triangle inequality we have
  $$\aligned
  \sum_{k \in \Z_0^d} A_{k} \Big( \sum_{j \in \Z_0^d} B_{j} C_{k-j} \Big)^{p_0}
  & = \sum_{k \in \Z_0^d} A_{k} \Big( \sum_{j \in \Z_0^d} B_{j} |k-j|^s |k-j|^{-s} C_{k-j} \Big)^{p_0} \\
  & \lesssim_{s} \sum_{k \in \Z_0^d} A_{k}^{(s)} \Big( \sum_{j \in \Z_0^d} B_{j} C_{k-j}^{(s)} \Big)^{p_0} + \sum_{k \in \Z_0^d} A_{k} \Big( \sum_{j \in \Z_0^d} B_{j}^{(s)} C_{k-j}^{(s)} \Big)^{p_0} .
  \endaligned $$
For the first term, using \eqref{eq-convolution sum} with $p=p_1$ and $\tilde{p}=\frac{p_0 p_1}{p_0 p_1-1}=:\tilde{p}_1$ we have
  $$ \sum_{k \in \Z_0^d} A_{k}^{(s)} \Big( \sum_{j \in \Z_0^d} B_{j} C_{k-j}^{(s)} \Big)^{p_0} \leq \big\|A^{(s)}\big\|_{\ell^{p_1}} \big\| B \big\|_{\ell^{\tilde{p}_1}}^{p_0} \big\|C^{(s)}\big\|_{\ell^{p_0}}^{p_0} .$$
Set $q_1 :=p_2 / \tilde{p}_1=p_2(1-\frac{1}{p_0 p_1}) \geq 1$; H\"older's inequality implies
  $$\aligned
   \big\| B \big\|_{\ell^{\tilde{p}_1}} &=\Big(\sum_{k \in \Z_0^d} B_k^{\tilde{p}_1} \Big)^{\frac{1}{\tilde{p}_1}} = \Big(\sum_{k \in \Z_0^d} |k|^{-s \tilde{p}_1} (|k|^{s} B_k )^{\tilde{p}_1} \Big)^{\frac{1}{\tilde{p}_1}} \\
  & \leq \Big(\sum_{k \in \Z_0^d} |k|^{-s \tilde{p}_1 \frac{q_1}{q_1-1}} \Big)^{\frac{q_1-1}{q_1\tilde{p}_1}} \Big(\sum_{k \in \Z_0^d} (|k|^{s} B_k )^{p_2} \Big)^{\frac{1}{p_2}} \lesssim \big\|B^{(s)}\big\|_{\ell^{p_2}},
  \endaligned $$
where the last step is due to $s \tilde{p}_1 \frac{q_1}{q_1-1}= s \big( 1-\frac{1}{p_0 p_1}-\frac{1}{p_2} \big)^{-1}>d $. So we get
  $$ \sum_{k \in \Z_0^d} A_{k}^{(s)} \Big( \sum_{j \in \Z_0^d} B_{j} C_{k-j}^{(s)} \Big)^{p_0} \lesssim \big\|A^{(s)}\big\|_{\ell^{p_1}} \big\|B^{(s)}\big\|_{\ell^{p_2}}^{p_0} \big\|C^{(s)}\big\|_{\ell^{p_0}}^{p_0} .$$

For the second term, applying \eqref{eq-convolution sum} with $p= \frac{p_2}{p_0(p_2 -1)} = \frac{p'_2}{p_0} \ge 1$ and $\tilde p = \frac{p_0 p}{p_0 p-1} = p_2$ yields
  $$\sum_{k \in \Z_0^d} A_{k} \Big( \sum_{j \in \Z_0^d} B_{j}^{(s)} C_{k-j}^{(s)} \Big)^{p_0} \leq \|A \|_{\ell^{p'_2/p_0}} \big\| B^{(s)} \big\|_{\ell^{p_2}}^{p_0} \big\| C^{(s)} \big\|_{\ell^{p_0}}^{p_0}. $$
Define $q_2=p_0p_1/p_2'=p_0p_1(1-\frac{1}{p_2})\geq 1$, by H\"older inequality, we have
  $$\|A \|_{\ell^{p'_2/p_0}} = \bigg(\sum_k |k|^{-sp'_2} (|k|^{p_0 s} A_k)^{\frac{p'_2}{p_0}} \bigg)^{\frac{p_0}{p'_2}} \le \Big(\sum_k |k|^{-sp'_2 \frac{q_2}{q_2-1}} \Big)^{\frac{p_0}{p'_2} \frac{q_2}{q_2-1}} \bigg(\sum_k \big(A^{(s)}_k \big)^{p_1} \bigg)^{\frac{1}{p_1}}.$$
By $sp'_2 \frac{q_2}{q_2-1} = s \big( 1-\frac{1}{p_0 p_1}-\frac{1}{p_2} \big)^{-1}>d $, we obtain $\|A \|_{\ell^{p'_2/p_0}} \lesssim \big\| A^{(s)} \big\|_{\ell^{p_1}}$. Thus, the second term enjoys the same estimate.
Combining these results we complete the proof.
\end{proof}

\begin{lemma} \label{cor-convolution sum}
  Let $a,b \in (0,d/2)$ and $0 \leq c<\min \{a,b,a+b-d/2\}$, for any $u \in H^{-c}(\T^d,\R^d)$, we have
  $$ \sum_{k,l \in \Z_0^d} |k|^{-2a} |l|^{-2b} |\< u, e_{k-l} \>|^2 \lesssim \| u\|_{H^{-c}}^2.$$
\end{lemma}

\begin{proof}
We will use Lemma \ref{lem-convolution sum} to prove this result.
Take $p_0=1$ and $s=2c$; let $A_k = |k|^{-2a} $, $B_k = |k|^{-2b}$ and $ C_k=|\< u, e_{k} \>|^2$, $k \in \Z_0^d$. Due to $a-c>0$, $b-c>0$ and $a-c+b>\frac{d}{2}$, we can find $p_1 \in \big(\frac{d}{2(a-c)}, \frac{d}{d-2b} \big) $ such that $A^{(2c)} \in \ell^{p_1}(\Z^d_0)$ and
  $$ 0< \eta(p_1):= 1-\frac{2c}{d}- \frac{1}{p_1} < \frac{2(b-c)}{d}.$$
Taking $p_2=1/\eta(p_1)$ when $c=0$, and $p_2< 1/\eta(p_1)$ when $c>0$, then condition \eqref{condition-convolution sum} holds. As $\eta(p_1)<\frac{2(b-c)}{d}$, we can further choose $p_2> \frac{d}{2(b-c)}$, then $B^{(2c)} \in \ell^{p_2}(\Z^d_0)$. By Lemma \ref{lem-convolution sum}, we have
$$  \sum_{k,l \in \Z_0^d} |k|^{-2a} |l|^{-2b} |\< u, e_{k-l} \>|^2 \lesssim \| C^{(2c)}\|_{\ell^1}= \| u\|_{H^{-c}}^2 .$$
We obtain the desired result.
\end{proof}

We need the following key estimates on the transport term $V\cdot\nabla f$. The estimate (i) below largely improves the second result in \cite[Lemma 2.1]{FGL21c} since we allow $b$ here to range in $(0,\frac d2)$.

\begin{lemma}\label{lem-nD-nonlinearity}
  For $d=2,3$, let $V\in H_\sigma$ be a divergence free vector field.
  \begin{itemize}
  \item[\rm(i)] Let $0<b<\frac{d}{2}<a$, $V\in H^a(\T^d,\R^d)$ and $f\in H^{-b}(\T^d)$, then
  \[
    \|V\cdot\nabla f\|_{H^{-1-b}} \lesssim_{a,b} \|V\|_{H^a} \|f\|_{H^{-b}}.
  \]
  \item[\rm(ii)] Let $b\in(0,\frac{d}{2})$, $V\in H^{\frac{d}{2}-b}(\T^d,\R^d)$ and $f\in L^{2}(\T^d)$, then
  \[
    \| V\cdot\nabla f\|_{H^{-1-b}} \lesssim_b \| V\|_{H^{\frac{d}{2}-b}} \, \|f\|_{L^{2}}.
  \]
  \item[\rm(iii)] Let $ 0 \leq b < \frac{d}{2}$, $\epsilon>0,$ $V \in H^{b}(\T^d,\R^d)$, and $f \in H^{-b}(\T^d)$, then
  \[
    \| V\cdot\nabla f\|_{H^{-1-\frac{d}{2}-\epsilon}} \lesssim_{b,\epsilon} \|V\|_{H^b} \|f\|_{H^{-b}}.
  \]
  \end{itemize}
\end{lemma}

\begin{proof}
We provide the proofs for completeness.

(i) By the divergence free assumption, we have $ \|V\cdot\nabla f\|_{H^{-1-b}} \lesssim \|Vf\|_{H^{-b}}$. Let $A=\{|k|^{-2b}\}_{k\in \Z^d_0}$, $B= \{ |\hat{V}_k| \}_{k\in \Z^d_0}$ and $C=\{ |\hat{f}_{k}| \}_{k\in \Z^d_0}$.
    Applying Lemma \ref{lem-convolution sum} with $s=b$, $p_0=2$, $p_1=+\infty$ and $p_2 < \frac{d}{d-b} <2 $, we have
$$ \aligned
  \|Vf\|_{H^{-b}}^2 & = \sum_{k \in \Z_0^d} |2 \pi k|^{-2b} \Big| \sum_{j \in \Z_0^d} \hat{V}_j \hat{f}_{k-j} \Big|^2 \lesssim \big\|A^{(b)}\big\|_{\ell^{\infty}} \big\|B^{(b)}\big\|_{\ell^{p_2}}^{2} \big\|C^{(b)}\big\|_{\ell^{2}}^{2} \\
  & \leq \Big( \sum_{k \in \Z_0^d} |k|^{bp_2} |\hat{V}_k|^{p_2} \Big)^{\frac{2}{p_2}} \|f\|_{H^{-b}}^2.
\endaligned $$
Due to $b<\frac{d}{2}<a$, for any $p_2 \in \big(\frac{2d}{d+2(a-b)}, \frac{d}{d-b} \big)$, one has $\big\{|k|^{-p_2(a-b)}\big\}_{k \in \Z_0^d} \in \ell^{\frac{2}{2-p_2}}(\Z_0^d)$. By H\"older inequality, it holds
  $$\|Vf\|_{H^{-b}}^2 \lesssim \Big( \sum_{k \in \Z_0^d} |k|^{-p_2 (a-b) \frac{2}{2-p_2}} \Big)^{\frac{2-p_2}{p_2}} \Big( \sum_{k \in \Z_0^d} |k|^{2a} |\hat{V}_k|^{2} \Big) \|f\|_{H^{-b}}^2 \lesssim \|V\|_{H^{a}}^2 \|f\|_{H^{-b}}^2.$$
(ii) Since $b<\frac{d}{2}$, we have $\frac{d}{2b}> 1$ and $\frac{d}{d-2b}< +\infty$. By the divergence free assumption on $V$ and H\"older inequality, it holds
    \[
    \aligned
    \|V\cdot\nabla f\|_{H^{-1-b}} = & \sup_{\|\phi\|_{H^{1+b}}=1} | \< V \cdot \nabla f,\phi \> |= \sup_{\|\phi\|_{H^{1+b}}=1} |\< f,V \cdot \nabla  \phi \> |\\
    \leq & \sup_{\|\phi\|_{H^{1+b}}=1} \|f \|_{L^2} \|V \cdot \nabla \phi \|_{L^2} \\
    \leq & \sup_{\|\phi\|_{H^{1+b}}=1} \|f \|_{L^2} \|V \|_{L^{\frac{d}{b}}} \|\nabla \phi \|_{L^{\frac{2d}{d-2b}}}.
    \endaligned
    \]
    By the Sobolev embeddings $\|V \|_{L^{\frac{d}{b}}} \lesssim \|V\|_{H^{\frac d2 -b}}$ and $\|\nabla \phi \|_{L^{\frac{2d}{d-2b}}} \lesssim \|\phi\|_{H^{b+1}}$, we have
    \[
    \|V\cdot\nabla f\|_{H^{-1-b}} \lesssim_{a} \|V\|_{H^{\frac{d}{2}-b}} \|f\|_{L^2}.
    \]
(iii) Since $V$ is divergence free, we have $\| V \cdot \nabla f\|_{H^{-1-\frac{d}{2}-\epsilon} } \lesssim \| V f\|_{H^{-\frac{d}{2}-\epsilon} } $. We have
    $$ \|Vf\|_{H^{-\frac{d}{2}-\eps}}^2 = \sum_{k \in \Z_0^d} |2 \pi k|^{-d-2\eps} \Big| \sum_{j \in \Z_0^d} \hat{V}_j \hat{f}_{k-j} \Big|^2 .$$
Let $A=\{|k|^{-d-2\eps}\}_{k\in \Z^d_0}$, $B= \{ |\hat{V}_k| \}_{k\in \Z^d_0}$ and $C=\{ |\hat{f}_{k}| \}_{k\in \Z^d_0}$. For $b=0$, we apply Lemma \ref{lem-convolution sum} with $s=b=0$, $p_0=2$, $p_1=1$ and $p_2=2$ to obtain
    $$ \|Vf\|_{H^{-\frac{d}{2}-\eps}}^2 \lesssim \|A \|_{\ell^{1}} \|B \|_{\ell^{2}}^{2} \|C \|_{\ell^{2}}^{2} \lesssim \| V\|_{L^2}^2 \| f\|_{L^2}^2.$$
If $b>0$, applying Lemma \ref{lem-convolution sum} with $s=b$, $p_0=2$, $p_1\in \big[1,\frac{d}{d-2b} \big)$ and $p_2=2$, we obtain
    $$ \|Vf\|_{H^{-\frac{d}{2}-\eps}}^2 \lesssim \big\|A^{(b)} \big\|_{\ell^{p_1}} \big\|B^{(b)} \big\|_{\ell^{2}}^{2} \big\|C^{(b)} \big\|_{\ell^{2}}^{2} = \Big(\sum_{k\in \Z_0^d} |k|^{-(d-b+ 2\eps)p_1} \Big)^{\frac{1}{p_1}} \| V\|_{H^{b}}^2 \| f\|_{H^{-b}}^2. $$
We can choose $p_1 \in \big(1\vee \frac{d}{d-b+ 2\eps} ,\frac{d}{d-2b} \big)$, then $\sum_{k\in \Z_0^d} |k|^{-(d-b+ 2\eps)p_1 }< +\infty$. So we obtain $\|Vf\|_{H^{-\frac{d}{2}-\eps}} \lesssim \| V\|_{H^{b}} \| f\|_{H^{-b}}$.
\end{proof}

\subsection{Convergence of the Stratonovich-It\^o correctors}\label{subs-correctors}
Recall the Stratonovich-It\^o correctors $S_\theta^{(d)}(\cdot)$ defined in \eqref{def-Stratonovich-Ito-correctors} and
the sequence $\theta^N \in \ell^2(\Z^d_0)$ in \eqref{theta-N-def}; we want to prove the convergence of $S^{(d)}_{\theta^N}(\cdot)$ and give the quantitative convergence rates. First we introduce the decreasing factor
\begin{equation}\label{def-D_N}
  D_N := \epsilon_N \sum_{1 \leq |k| \leq N} \frac{1}{|k|^{2\gamma+1}} \sim \begin{cases}
    N^{2\gamma -d}, & \gamma\in \big( \frac{d-1}{2}, \frac{d}{2} \big) ;\\
    \frac{\log N}{N}, & \gamma=\frac{d-1}{2} ;\\
    N^{-1}, & \gamma\in \big(0, \frac{d-1}{2} \big).
  \end{cases}
\end{equation}
Here is the relation between $D_N$ and $\| \theta^N \|_{\ell^\infty} = \sqrt{\epsilon_N}$.

\begin{proposition}\label{pro-D_N}
For any $q \in (0,1 \wedge \frac{1}{d-2\gamma})$, we have $ D_N \lesssim_{q} \epsilon_N^{q}=\| \theta^N \|_{\ell^\infty}^{2q}$.
\end{proposition}

\begin{proof}
H\"older's inequality yields
  \begin{equation*}
    D_N = \epsilon_N^{q} \sum_{k} (\theta_k^{N})^{2(1-q)} \frac{1}{|k|^{2q\gamma+1}} \leq \epsilon_N^{q} \Bigl( \sum_k (\theta_k^{N})^{2} \Bigr)^{1-q} \biggl( \sum_k \frac{1}{|k|^{2\gamma+ 1/q}} \biggr)^{q}.
  \end{equation*}
As $q<\frac{1}{d-2\gamma}$, it holds $\sum_k \frac{1}{|k|^{2\gamma+ 1/q}}  < +\infty$. So we have
  \begin{equation*}
    D_N \lesssim_{q} \epsilon_N^{q} \Bigl( \sum_k (\theta_k^{N})^{2} \Bigr)^{1-q} =\epsilon_N^{q}.
  \end{equation*}
In the last step, we used the fact $\sum_k (\theta_k^{N})^{2}=1$.
\end{proof}

Now we are ready to present two key limit results which are generalizations of \cite[Theorem 5.1]{FL21} and \cite[Theorem 3.1]{Luo21b}; the main difference from these two results lies in the choice of coefficients $\theta^N$, because we allow here that $\theta^N$ does not vanish for lower modes. As the proofs are quite technical we postpone them to Appendix \ref{sec-appendix-B}.

\begin{theorem}\label{thm-Ito-corrector}
Let $\theta^N$ be defined as in \eqref{theta-N-def}, $N\geq 1$. There exists a constant $C>0$, independent of $N\geq 1$, such that for any $\alpha\in [0,1]$, and
\begin{itemize}
\item (2D case) for any divergence free field $v\in H^b(\T^2,\R^2)$, it holds
  \begin{equation}\label{thm-Ito-corrector.1}
  \bigg\| S_{\theta^N}^{(2)} (v)- \frac{1}{4} \kappa \Delta v \bigg\|_{H^{b-2-\alpha}} \leq C \kappa D_N^{\alpha} \|v \|_{H^b};
  \end{equation}
\item (3D case) for any divergence free field $v\in H^b(\T^3,\R^3)$, it holds
  \begin{equation}\label{thm-Ito-corrector.2}
  \bigg\| S_{\theta^N}^{(3)} (v)- \frac{3}{5} \kappa \Delta v \bigg\|_{H^{b-2-\alpha}} \leq C \kappa D_N^{\alpha} \|v \|_{H^b}.
  \end{equation}
\end{itemize}
\end{theorem}

\subsection{Maximal estimates on stochastic convolution}

We first state some properties of the semigroup generated by the Stokes operator.

\begin{lemma}\label{lem-semigroup}
Let $a \in \R$.
\begin{itemize}
  \item [\rm(i)] For any $\varphi\in H^a_\sigma(\T^d,\R^d)$ and $\delta \geq 0$, it holds $\|e^{t \Delta} \varphi \|_{H^{a+\delta}} \lesssim_{\delta} t^{-\delta/2} \| \varphi \|_{H^{a}}$.
  \item [\rm(ii)] For any $\varphi\in H^a_\sigma(\T^d,\R^d)$ and $\delta \in [0,2]$, it holds  $\|(I-e^{t \Delta} )\varphi \|_{H^{a-\delta}} \lesssim_{\delta} t^{\delta/2} \| \varphi \|_{H^{a}}$.
  \item [\rm(iii)] If $\mu>0$ and $\psi \in L^2(0,T; H^a_\sigma (\T^d,\R^d))$, then
  $$\bigg\|\int_0^t e^{\mu (t-s) \Delta} \psi_s\,\d s \bigg\|_{H^{a+1}}^2 \lesssim \frac1\mu \int_0^t \|\psi_s \|_{H^a}^2\,\d s. $$
\end{itemize}
\end{lemma}

Let $\{f_t \}_{t\in [0,T]}$ be a progressively measurable process of vector fields on $\T^d$, such that $\P$-a.s.,
  $$\sup_{t\in [0,T]} \|f_t \|_{L^2} \leq R. $$
Given $\mu>0$, we want to find some estimates on the stochastic convolution
  $$Z_t= \sqrt{C_d\, \kappa}\sum_{k,i} \theta_k \int_0^t e^{\mu(t-s) \Delta} \Pi(\sigma_{k,i} \cdot\nabla f_s) \,\d W^{k,i}_s.$$
The following result was first proved in \cite[Corollary 2.6]{FGL21c} in the case that $f$ is a process of functions, but the same proof works for vector field valued processes with few changes.

\begin{proposition} \label{prop-stoch-convol}
For any $\beta \in (0, d/2]$, $p\geq 1$ and any $\delta \in (0,\beta)$, it holds
  $$\E \bigg[\sup_{t\in [0,T]} \|Z_{t} \|_{H^{-\beta}}^p \bigg]^{1/p} \lesssim_{\delta,p,T} \sqrt{\kappa\mu^{\delta-1} } \|\theta \|_{\ell^\infty}^{2(\beta-\delta)/d} R. $$
\end{proposition}

Next, we give a slightly better estimate of stochastic convolution for the special noise with coefficients \eqref{theta-N-def}, in the case $\gamma>\frac{d-2}{2}$. This result will be uses in the proof of the second estimate in Theorem \ref{thm-main}. We write $Z_t^N$ if the coefficient $\theta$ is replaced by $\theta^N$.

\begin{proposition}\label{prop-stoch-convol-1}
  Let $\gamma>\frac{d-2}{2}$. For any $\beta \in (0, 1]$, $p \geq 1$ and $\delta \in (0,\beta )$, it holds
  $$\E \bigg[\sup_{t\in [0,T]} \|Z_{t}^N \|_{H^{-\beta}}^p \bigg]^{1/p} \lesssim_{\delta,p,T} \sqrt{\kappa\mu^{\delta-1} } \epsilon_N^{(\beta-\delta)/2} R.$$
\end{proposition}

\begin{proof}
We follow the idea of proof of \cite[Corollary 2.6]{FGL21c}. For $\alpha>0$, by Burkholder-Davis-Gundy's inequality, we have
  \begin{align*}
  \E\big[ \| Z_t^N \|_{H^{-\alpha}}^{p} \big]^{\frac{1}{p}}
  &= \sqrt{C_d \kappa\, \epsilon_N}\, \E\bigg[ \Big\| \sum_{|k|\leq N} \sum_{i=1}^{d-1} \frac{1}{|k|^{\gamma}} \int_0^t e^{\mu(t-s) \Delta} \Pi (\sigma_{k,i} \cdot\nabla f_s)\, \d W^{k,i}_s \Big\|_{H^{-\alpha}}^{p} \bigg]^{\frac{1}{p}} \\
  & \lesssim \sqrt{\kappa\, \epsilon_N}\, \E\bigg[ \Big( \sum_{|k|\leq N} \sum_{i=1}^{d-1} |k|^{-2\gamma} \int_0^t \| e^{\mu(t-s) \Delta} \Pi (\sigma_{k,i} \cdot\nabla f_s) \|_{H^{-\alpha}}^2 \,\d s \Big)^{p/2} \bigg]^{\frac{1}{p}} \\
  & \lesssim \sqrt{\kappa\, \epsilon_N \mu^{\delta-1} }\, \E\bigg[ \Big( \sum_{|k|\leq N} \sum_{i=1}^{d-1} |k|^{-2\gamma} \int_0^t (t-s)^{\delta-1} \| \sigma_{k,i} \cdot\nabla f_s \|_{H^{-\alpha-1+\delta}}^2 \,\d s \Big)^{p/2} \bigg]^{\frac{1}{p}},
  \end{align*}
where in the last step we used Lemma \ref{lem-semigroup}(i). By the definition of $\sigma_{k,i}$, we know
  \begin{equation} \label{norm-transport-term}
  \| \sigma_{k,i} \cdot\nabla f_s \|_{H^{-\alpha-1+\delta}} =\| \nabla \cdot (\sigma_{k,i} \otimes f_s) \|_{H^{-\alpha-1+\delta}}  \lesssim \|e_k \, f_s\|_{H^{-\alpha+\delta}}.
  \end{equation}
Thus, we have
  $$ \aligned
  \E\big[ \|Z_t^N \|_{H^{-\alpha}}^{p} \big]^{\frac{1}{p}}
  & \lesssim \sqrt{\kappa\, \epsilon_N \mu^{\delta-1}t^{\delta}}\, \E\bigg[ \sup_{s\in [0,T]} \Big( \sum_{|k|\leq N} \sum_{i=1}^{d-1}|k|^{-2\gamma} \| \sigma_{k,i} \cdot\nabla f_s \|_{H^{-\alpha-1+\delta}}^2 \Big)^{p/2} \bigg]^{\frac{1}{p}} \\
  & \lesssim \sqrt{\kappa\, \epsilon_N \mu^{\delta-1} t^{\delta}}\, \E\bigg[ \sup_{s\in[0,T]} \Big( \sum_{|k|\leq N}  \sum_{l \in \Z_0^d} |k|^{-2\gamma} |l|^{-2(\alpha-\delta)} |\<f_s, e_{k-l}\>|^2 \Big)^{p/2} \bigg]^{\frac{1}{p}} .
  \endaligned $$
Taking $\alpha=1+\delta$, due to $\gamma+1>\frac{d}{2}$, we can apply Lemma \ref{cor-convolution sum} with $a=\gamma$, $b=1$ and $c=0$ to obtain
  \begin{equation} \label{proof-convol-1}
    \E\big[ \|Z_t^N \|_{H^{-\delta-1}}^{p} \big]^{\frac{1}{p}} \lesssim \sqrt{\kappa\, \epsilon_N \mu^{\delta-1} t^{\delta}} \, \E \bigg[\sup_{t \in [0,T]} \|f_t\|_{L^2}^p \bigg]^{\frac{1}{p}} \leq \sqrt{\kappa \epsilon_N \mu^{\delta-1} t^{\delta}}\, R.
  \end{equation}

Next, observing that by construction $Z^N$ satisfies
  $$ Z_t^N=e^{\mu(t-s) \Delta} Z_s^N +\sqrt{C_d\, \kappa}\sum_{k,i} \theta_k \int_s^t e^{\mu(t-r) \Delta} \Pi(\sigma_{k,i} \cdot\nabla f_r) \,\d W^{k,i}_r.$$
  We denote the second term on the right-hand side as $Z_{s,t}^N$. By Lemma \ref{lem-semigroup}(ii), we have
  \begin{align} \label{proof-convol-Z}
    \|Z_t^N-Z_s^N\|_{H^{-2\delta-1}} & \lesssim \big\| \big(I-e^{\mu (t-s)\Delta}\big)Z_s^N \big\|_{H^{-2\delta-1}} +\big\| Z_{s,t}^N \big\|_{H^{-2\delta-1}} \nonumber \\
    & \lesssim \mu^{\delta/2} |t-s|^{\delta/2} \big\| Z_s^N \big\|_{H^{-\delta-1}} + \big\| Z_{s,t}^N \big\|_{H^{-2\delta-1}}.
  \end{align}
Similar to the estimate \eqref{proof-convol-1}, we have the estimate of $Z_{s,t}^N$
  $$ \E \big[ \big\| Z_{s,t}^N \big\|_{H^{-2\delta-1}}^{2p} \big]^{\frac{1}{2p}} \lesssim \sqrt{\kappa\, \epsilon_N \mu^{2\delta-1}}\, |t-s|^{\delta} R.$$
Taking expectation on both sides of the inequality \eqref{proof-convol-Z} and by the previous estimates we obtain
  $$ \E \Big[ \big\| Z_{t}^N-Z_{s}^N\big\|_{H^{-2\delta-1}}^{2p} \Big]^{\frac{1}{2p}} \lesssim \sqrt{\kappa\, \epsilon_N \mu^{2\delta-1}}\, |t-s|^{\delta/2} R .$$
Renaming $2\delta$ as $\delta$ give us
  $$ \E \Big[ \big\| Z_{t}^N-Z_{s}^N\big\|_{H^{-\delta-1}}^{2p} \Big] \lesssim \big(\sqrt{\kappa\, \epsilon_N \mu^{\delta-1}} R \big)^{2p} |t-s|^{p \delta/2}.$$
Choosing $p> \max\{1,\frac{2}{\delta}\}$ and applying Kolmogorov's continuity criterion, we obtain
  \begin{equation} \label{proof-convol-2}
    \E\bigg[ \sup_{t \in [0,T]} \|Z_t^N \|_{H^{-\delta-1}}^{p} \bigg]^{\frac{1}{p}} \lesssim \sqrt{\kappa\, \epsilon_N \mu^{\delta-1}}\, R.
  \end{equation}
Similar to the proof of \cite[Corollary 2.6]{FGL21c}, we have
  \begin{equation} \label{proof-convol-3}
    \E\bigg[ \sup_{t\in [0,T]} \|Z_t^N \|_{H^{-\delta}}^{p} \bigg]^{\frac{1}{p}} \lesssim \sqrt{\kappa \mu^{\delta-1}} R.
  \end{equation}
  Setting $\lambda=\beta-\delta$, then $\lambda \in (0,1)$ and $\beta=\lambda(1+\delta)+(1-\lambda)\delta$. By interpolating between \eqref{proof-convol-2} and \eqref{proof-convol-3}, we obtain the desired result.
\end{proof}

\section{Quantitative convergence rates} \label{sec-proof-converg-rates}

The purpose of this section is to prove Theorem \ref{thm-main}. We first briefly recall the setting: we consider the following equations
  \begin{equation} \label{regular-approx-eq}
  \d u^N + \Pi(v^N \cdot \nabla u^N)\,\d t =\sqrt{C_{d} \kappa} \, \sum_{k,i} \theta_{k}^N \Pi(\sigma_{k,i} \cdot \nabla u^N) \, \d W^{k,i}_t + S_{\theta^N}^{(d)}(u^N) \, \d t
  \end{equation}
with $v^N= K(u^N)= (1+ (-\Delta)^{\gamma_0})^{-1} (u^N)$ and $u^N|_{t=0}= u_0 \in H_\sigma$. Because of Theorem \ref{thm-Ito-corrector}, we rewrite \eqref{regular-approx-eq} as
  $$\aligned
  \d u^N + \Pi(v^N \cdot \nabla u^N) \,\d t &= \sqrt{C_{d} \kappa} \, \sum_{k,i} \theta_{k}^N \Pi(\sigma_{k,i} \cdot \nabla u^N) \, \d W^{k,i}_t \\
  &\quad + \Big(S_{\theta^N}^{(d)} (u^N)- C_d^{\prime} \kappa\Delta u^N \Big)\,\d t +  C_d^{\prime} \kappa\Delta u^N \,\d t,
  \endaligned $$
where $C'_d$ is defined in \eqref{C-d-prime}. Denote $P_t= e^{C_d^{\prime} \kappa t\Delta}\, (t\geq 0)$ the semigroup generated by the Stokes operator. We write the above equation in mild form as
  $$u^N_t = P_t u_0 - \int_0^t P_{t-s} \Pi(v^N_s\cdot \nabla u^N_s)\,\d s + \int_0^t P_{t-s} \Big(S_{\theta^N}^{(d)} (u^N_s)- C_d^{\prime} \kappa\Delta u^N_s\Big)\,\d s + Z^N_t, $$
where $Z_t^N$ is the stochastic convolution
  $$Z_t^N= \sqrt{C_d \kappa} \sum_{k,i} \theta_k^N \int_0^t P_{t-s} \Pi(\sigma_{k,i} \cdot \nabla u^N_s) \, \d W^{k,i}_s. $$
By Theorem \ref{thm-Ito-corrector}, we expect equation \eqref{regular-approx-eq} will tend to the limit equation \eqref{regular-limit-eq}:
  $$\partial_t \tilde u + \Pi(\tilde v\cdot \nabla \tilde u) = C'_d \kappa \Delta \tilde u $$
with $\tilde v= K(\tilde u)$ and $\tilde u_0 = u_0$. In mild form, the equation reads as
  $$\tilde u_t = P_t u_0 - \int_0^t P_{t-s} \Pi(\tilde v_s\cdot \nabla\tilde u_s)\,\d s .$$

Now we start proving Theorem \ref{thm-main}. For simplicity, we assume $\kappa\geq 1$; the argument will be slightly different when $0<\kappa<1$. From the mild formulations one deduces
  $$u_t^N - \tilde u_t = -\int_0^t P_{t-s} \Pi (v_s^N \cdot \nabla u_s^N -\tilde v_s\cdot \nabla\tilde u_s) \,\d s + \int_0^t P_{t-s} \Big(S_{\theta^N}^{(d)} (u_s^N)- C_d^{\prime} \kappa\Delta u_s^N \Big)\,\d s + Z_t^N, $$
and thus
  \begin{equation}\label{main-difference}
  \aligned
  \|u_t^N -\tilde u_t\|_{H^{-\alpha}} &\leq \bigg\| \int_0^t P_{t-s} \Pi (v_s^N \cdot \nabla u_s^N -\tilde v_s\cdot \nabla\tilde u_s) \,\d s \bigg\|_{H^{-\alpha}} \\
  &\quad + \bigg\| \int_0^t P_{t-s} \Big(S_{\theta^N}^{(d)} (u_s^N)- C_d^{\prime} \kappa\Delta u_s^N \Big)\,\d s \bigg\|_{H^{-\alpha}} + \|Z_t^N \|_{H^{-\alpha}}.
  \endaligned
  \end{equation}
We denote the first two terms on the right-hand side by $I_{1,N} (t)$ and $I_{2,N}(t)$, respectively.

\begin{lemma}\label{lem-main-nonlin}
  For $\gamma_0 \in (\frac{d-2}{4},\frac{d+2}{4})$, $\alpha \in (0, (2\gamma_0) \wedge \frac{d}{2})$ and $q>\max\{2,\frac{4}{4\gamma_0-d+2}\}$, we have
    $$ I_{1,N}^q(t) \lesssim_{T.q,\gamma_0} \|u_0 \|_{L^2}^{q-2} \int_0^t \bigg(\frac{\|u_0 \|_{L^2}^2}{\kappa^{q(\frac{d+2}{4}-\gamma_0)}} + \frac{\|\tilde u_s \|_{H^1}^2}{\kappa^{q-1}} \bigg)\|u_s^N -\tilde u_s\|_{H^{-\alpha}}^q\,\d s. $$
\end{lemma}

\begin{proof}
  We follow the idea of proof of \cite[Proposition 3.4]{FGL21c}. It holds that
    $$\aligned
    I_{1,N}(t)&\leq \bigg\| \int_0^t P_{t-s} \Pi ((v_s^N-\tilde v_s)\cdot \nabla u_s^N) \,\d s \bigg\|_{H^{-\alpha}} + \bigg\| \int_0^t P_{t-s} \Pi (\tilde v_s \cdot \nabla (u_s^N-\tilde u_s)) \,\d s \bigg\|_{H^{-\alpha}} \\
    &=: I_{1,1,N}(t) +  I_{1,2,N}(t) .
    \endaligned $$
  First, for some $\eps_1\in (\frac{1}{q},1)$ applying Lemma \ref{lem-semigroup}(i) with $\delta=2-2\eps_1> 0$, we have
    $$\aligned
    I_{1,1,N}(t) &\leq \int_0^t \big\| P_{t-s} \Pi ((v_s^N-\tilde v_s)\cdot \nabla u_s^N) \big\|_{H^{-\alpha}} \,\d s \\
    &\lesssim \int_0^t \frac1{(\kappa(t-s))^{1-\eps_1} }\big\| \Pi ((v_s^N-\tilde v_s)\cdot \nabla u_s^N) \big\|_{H^{-\alpha-2+2\eps_1}} \,\d s \\
    &\leq \frac1{\kappa^{1-\eps_1}} \int_0^t \frac1{(t-s)^{1-\eps_1}} \big\| (v_s^N-\tilde v_s)\cdot \nabla u_s^N \big\|_{H^{-\alpha-2+2\eps_1}} \,\d s,
    \endaligned $$
  where the last step is due to that, for any $a\in \R$, the operator norm of $\Pi$ in $H^{a}$ is less than 1.

  If $\frac{d-2}{4}<\gamma_0<\frac{d}{4}$, let $\eps_1 =\gamma_0 -\frac{d-2}{4}>\frac{1}{q}>0$.
  Applying Lemma \ref{lem-nD-nonlinearity}(ii) with $b= \alpha+1-2\eps_1 =\alpha +\frac{d}{2}-2\gamma_0 \in (0,\frac{d}{2})$, and noting that $\frac{d}{2} -b= 2\gamma_0 -\alpha$, we get
    $$\aligned
    I_{1,1,N}(t) &\lesssim \frac{1}{\kappa^{1-\eps_1}} \int_0^t \frac{1}{(t-s)^{1-\eps_1}} \|v_s^N-\tilde v_s \|_{H^{2\gamma_0-\alpha}} \|u_s^N \|_{L^2} \,\d s \\
    &\leq \frac{\|u_0 \|_{L^2}}{\kappa^{1-\eps_1}} \int_0^t \frac{1}{(t-s)^{1-\eps_1}} \|u_s^N-\tilde u_s \|_{H^{-\alpha}} \,\d s,
    \endaligned $$
  where we have used the fact that $\P$-a.s. for all $s\geq 0$, $\|u_s^N \|_{L^2}\leq \|u_0 \|_{L^2}$.

  If $ \frac{d}{4} \leq \gamma_0 < \frac{d+2}{4}$, let $\eps_1 =\frac{1}{2}>\frac{1}{q}>0$.
  Applying Lemma \ref{lem-nD-nonlinearity}(ii) with $b= \alpha+1-2\eps_1 =\alpha \in (0,\frac{d}{2})$, and noting that $\frac{d}{2} -b=\frac{d}{2}-\alpha \leq  2\gamma_0 -\alpha$, we get
    $$\aligned
    I_{1,1,N}(t) &\lesssim \frac{1}{\kappa^{1-\eps_1}} \int_0^t \frac{1}{(t-s)^{1-\eps_1}} \|v_s^N-\tilde v_s \|_{H^{\frac{d}{2}-\alpha}} \|u_s^N \|_{L^2} \,\d s \\
    &\lesssim \frac{1}{\kappa^{1-\eps_1}} \int_0^t \frac{1}{(t-s)^{1-\eps_1}} \|v_s^N-\tilde v_s \|_{H^{2\gamma_0-\alpha}} \|u_s^N \|_{L^2} \,\d s ,
    \endaligned $$
  where in the last step we have used Poincare inequality. By the fact $\|u_s^N \|_{L^2}\leq \|u_0 \|_{L^2}$ $\P$-a.s. for all $s\geq 0$, we have
    $$
    I_{1,1,N}(t) \leq \frac{\|u_0 \|_{L^2}}{\kappa^{1-\eps_1}} \int_0^t \frac{1}{(t-s)^{1-\eps_1}} \|u_s^N-\tilde u_s \|_{H^{-\alpha}} \,\d s.
    $$

  Combining above discussions and by H\"older's inequality with exponents $\frac1q+ \frac1{q'}=1$, it holds
    $$\aligned
    I_{1,1,N}(t) &\lesssim \frac{\|u_0 \|_{L^2}}{\kappa^{1-\eps_1}}  \bigg( \int_0^t \frac{\d s}{(t-s)^{(1-\eps_1)q'}} \bigg)^{1/q'} \bigg( \int_0^t \|u_s^N-\tilde u_s \|_{H^{-\alpha}}^q \,\d s \bigg)^{1/q} \\
    &\lesssim_{T,q,\gamma_0}\frac{\|u_0 \|_{L^2}}{\kappa^{1-\eps_1}} \bigg( \int_0^t \|u_s^N -\tilde u_s \|_{H^{-\alpha}}^q \,\d s \bigg)^{1/q}.
    \endaligned $$
  The second step is due to $(1-\eps_1)q'<1$ which is a consequence of $\eps_1=\min \{\frac{1}{2},\gamma_0-\frac{d-2}{4}\} >\frac{1}{q}$.

  Now we turn to estimate the second term. By Lemma \ref{lem-semigroup}(iii), we have
    $$\aligned
    I_{1,2,N}(t) &\lesssim \bigg( \frac1\kappa \int_0^t \big\| \Pi (\tilde v_s \cdot \nabla (u_s^N-\tilde u_s))\big\|_{H^{-\alpha-1}}^2 \,\d s \bigg)^{1/2} \\
    &\le \frac1{\sqrt \kappa} \bigg( \int_0^t \| \tilde v_s \cdot \nabla (u_s^N-\tilde u_s) \|_{H^{-\alpha-1}}^2 \,\d s \bigg)^{1/2} .
    \endaligned $$
  Applying the first result in Lemma \ref{lem-nD-nonlinearity}(i) with $a=1+2\gamma_0>\frac{d}{2}$ and $b=\alpha<\frac{d}{2}$, we arrive at
    $$\aligned
    I_{1,2,N}(t) &\lesssim \frac1{\sqrt \kappa} \bigg( \int_0^t \| \tilde v_s\|_{H^{1+2\gamma_0}}^2 \|u_s^N-\tilde u_s \|_{H^{-\alpha}}^2 \,\d s \bigg)^{1/2} \\
    &\lesssim \frac1{\sqrt \kappa} \bigg( \int_0^t \| \tilde u_s\|_{H^{1}}^2 \|u_s^N -\tilde u_s \|_{H^{-\alpha}}^2 \,\d s \bigg)^{1/2}.
    \endaligned $$
  For $q>2$, H\"older's inequality yields
    $$\aligned
    I_{1,2,N}(t) &\lesssim \frac1{\sqrt \kappa} \bigg( \int_0^t \| \tilde u_s\|_{H^{1}}^2 \,\d s \bigg)^{\frac{q-2}{2q}} \bigg( \int_0^t \| \tilde u_s \|_{H^{1}}^2 \|u^N_s-\tilde u_s \|_{H^{-\alpha}}^q \,\d s \bigg)^{1/q} \\
    &\lesssim_{T,q} \frac{\|u_0 \|_{L^2}^{1-2/q}}{\kappa^{1-1/q}} \bigg( \int_0^t \| \tilde u_s\|_{H^{1}}^2 \|u_s^N -\tilde u_s \|_{H^{-\alpha}}^q \,\d s \bigg)^{1/q} ,
    \endaligned $$
  where in the second step we have used the estimate \eqref{energy-ineq}.  Combining the above estimates we finish the proof.
\end{proof}

\begin{lemma}\label{lem-main-nonlin-2}
  For $\alpha \in (0,1]$ and any $0 < \delta <\alpha$, we have
    \begin{equation} \label{eq:lem-main-nonlin-2.1}
      I_{2,N}(t) \lesssim_{\delta,T} \kappa^{\delta/2} \epsilon_N^{(\alpha-\delta)/d} \|u_0\|_{L^2}.
    \end{equation}
  Furthermore, if $\gamma>\frac{d-2}{2}$, we have
  \begin{equation} \label{eq:lem-main-nonlin-2.2}
      I_{2,N}(t) \lesssim_{\delta,T} \kappa^{\delta/2} \epsilon_N^{(\alpha-\delta)/2} \|u_0\|_{L^2}.
  \end{equation}
  \end{lemma}

  \begin{proof}
  By the semigroup property, for some $0<\delta'<\delta<\alpha$,
    $$\aligned
    I_{2,N}(t)&\leq \int_0^t \Big\| P_{t-s} \Big(S_{\theta^N}^{(d)}(u_s^N)- C_d^{\prime} \kappa\Delta u_s^N \Big)\Big\|_{H^{-\alpha}} \,\d s \\
    &\lesssim \int_0^t \frac1{(\kappa(t-s))^{1-\delta'/2}} \Big\| S_{\theta^N}^{(d)} (u_s^N)- C_d^{\prime} \kappa\Delta u_s^N \Big\|_{H^{-\alpha-2+\delta'}} \,\d s .
    \endaligned $$
  Applying Theorem \ref{thm-Ito-corrector} with $b=0$ and $\alpha-\delta'$ in place of $\alpha$, we obtain
    $$\aligned
    I_{2,N}(t)&\lesssim  C\kappa D_N^{\alpha-\delta'} \int_0^t \frac{\|u_s^N \|_{L^2}}{(\kappa(t-s))^{1-\delta'/2}}  \,\d s
    \leq C \kappa^{\delta'/2} D_N^{\alpha-\delta'} \int_0^t \frac{\|u_0\|_{L^2}}{(t-s)^{1-\delta'/2}} \,\d s \\
    & \lesssim_{\delta',T} \kappa^{\delta'/2} D_N^{\alpha-\delta'} \|u_0\|_{L^2} \leq \kappa^{\delta/2} D_N^{\alpha-\delta'} \|u_0\|_{L^2}.
    \endaligned $$
Since $0<\delta'<\delta<\alpha$, we have $\frac{\alpha-\delta}{\alpha-\delta'}<1$.
For general case, we apply Proposition \ref{pro-D_N} with $q=\frac{\alpha-\delta}{d(\alpha-\delta')}$ to obtain $D_N^{\alpha-\delta'} \lesssim \epsilon_N^{(\alpha-\delta)/d}$.
For $\gamma>\frac{d-2}{2}$, applying Proposition \ref{pro-D_N} with $q=\frac{\alpha-\delta}{2(\alpha-\delta')} <\frac{1}{d-2\gamma}$, we have $D_N^{\alpha-\delta'} \lesssim \epsilon_N^{(\alpha-\delta)/2}$. Combining these we get the desired result.
\end{proof}

Finally, we are ready to prove Theorem \ref{thm-main}.

\begin{proof}[Proof of Theorem \ref{thm-main}]
By \eqref{main-difference} and Lemma \ref{lem-main-nonlin}, for any $q>\max\{2,\frac{4}{4\gamma_0-d+2}\}$ we have
  $$\|u_t^N -\tilde u_t\|_{H^{-\alpha}}^q \lesssim_{T,q,\gamma_0} \|u_0 \|_{L^2}^{q-2} \int_0^t \bigg(\frac{\|u_0 \|_{L^2}^2}{\kappa^{q(\frac{d+2}{4}-\gamma_0)}} + \frac{\|\tilde u_s \|_{H^1}^2}{\kappa^{q-1}} \bigg)\|u^N_s -\tilde u_s\|_{H^{-\alpha}}^q\,\d s + I_{2,N}^q(t) + \|Z_t^N \|_{H^{-\alpha}}^q.$$
Gronwall's inequality implies
  $$ \sup_{t\leq T} \|u_t^N -\tilde u_t\|_{H^{-\alpha}}^q \lesssim \sup_{t\leq T} \big[I_{2,N}^q(t)+ \|Z_t^N \|_{H^{-\alpha}}^q \big] \exp\bigg[ C \|u_0 \|_{L^2}^{q-2}\!\! \int_0^T \! \!\bigg(\frac{\|u_0 \|_{L^2}^2}{\kappa^{q(\frac{d+2}{4}-\gamma_0)}} + \frac{\|\tilde u_s \|_{H^1}^2}{\kappa^{q-1}} \bigg) \d s \bigg],$$
where the constant $C=C(T,q,\gamma_0)>0$. Using the energy estimate \eqref{energy-ineq} of $\tilde u$, we have
  $$\int_0^T\! \!\bigg(\frac{\|u_0 \|_{L^2}^2}{\kappa^{q(\frac{d+2}{4}-\gamma_0)}} + \frac{\|\tilde u_s \|_{H^1}^2}{\kappa^{q-1}} \bigg) \d s \lesssim \frac{\|u_0 \|_{L^2}^2}{\kappa^{q(\frac{d+2}{4}-\gamma_0)}} T + \frac{\|u_0 \|_{L^2}^2}{\kappa^{q}} = \frac{\|u_0 \|_{L^2}^2}{\kappa^{q}} \big(\kappa^{q (\gamma_0-\frac{d-2}{4})} T+1 \big), $$
which is a deterministic quantity. Inserting this estimate into the above inequality leads to
  \begin{equation} \label{eq:proof-main}
    \sup_{t\leq T} \|u_t^N -\tilde u_t\|_{H^{-\alpha}}^q \lesssim \bigg[ \sup_{t\leq T} I_{2,N}^q(t) + \sup_{t\leq T} \|Z_t^N \|_{H^{-\alpha}}^q \bigg] \exp\bigg[ C \frac{\|u_0 \|_{L^2}^q}{\kappa^{q}} \big(\kappa^{q (\gamma_0-\frac{d-2}{4})} T+1 \big) \bigg].
  \end{equation}

In the general case $\gamma\in (0,\frac d2)$, by \eqref{eq:lem-main-nonlin-2.1} and applying Proposition \ref{prop-stoch-convol} with $\beta=\alpha<\frac{d}{2}$, $p=q$, $\delta \in (0,\alpha)$ and $R=\|u_0 \|_{L^2}$, we obtain
  $$\E \bigg[ \sup_{t\leq T} I_{2,N}^q(t) \bigg]+ \E\bigg[\sup_{t\leq T} \|Z_t^N \|_{H^{-\alpha}}^q \bigg] \lesssim_{\delta, q, T} \kappa^{q\delta/2} \epsilon_N^{q(\alpha-\delta)/d} \|u_0\|_{L^2}^q .$$
Taking expectation on both sides of the inequality \eqref{eq:proof-main}, we obtain the desired result.

For $\gamma>\frac{d-2}{2}$, using Proposition \ref{prop-stoch-convol-1} with $\beta=\alpha<1$, $p=q$, $\delta \in (0,\alpha)$ and $R=\|u_0 \|_{L^2}$ and by the estimate of $I_{2,N}$ as \eqref{eq:lem-main-nonlin-2.2}, we can get the corresponding estimate.
\end{proof}

\section{Central limit theorem} \label{sec-CLT}

This section is devoted to the proof of the CLT with explicit rate of strong convergence. We start again by recalling the setting: $u^N$ and $\tilde{u}$ are solutions to equations \eqref{eq-inviscid-Leray} and \eqref{regular-limit-eq}, respectively; set
  $$U^N = (u^N-\tilde{u} ) / \sqrt{\epsilon_N},$$
which satisfies the following equation in a weak sense
  \begin{equation} \label{CLT-approx-eq}
  \aligned
  \d U^N + \Pi \big( V^N \cdot \nabla u^N + \tilde{v} \cdot \nabla U^N \big) \, \d t
  &= C_d^{\prime}\kappa \Delta U^N \, \d t + \frac{1}{\sqrt{\epsilon_N}} \big( S_{\theta^N}^{(d)} (u^N)-C_d^{\prime}\kappa \Delta u^N \big) \, \d t \\
   &\quad + \sqrt{C_d \kappa} \sum_{|k| \leq N} \sum_{i=1}^{d-1} \frac{1}{|k|^{\gamma}} \Pi \bigl( \sigma_{k,i} \cdot \nabla u^N \bigr)\, \d W_{t}^{k,i},
  \endaligned
  \end{equation}
where $V^N = K(U^N)= (v^N-\tilde{v} ) / \sqrt{\epsilon_N}$. Letting $N \to \infty$, the limit equation would be
\begin{equation} \label{CLT-limit-eq}
  \d U  + \Pi \left( V \cdot \nabla \tilde{u} + \tilde{v} \cdot \nabla U \right) \d t =C_d^{\prime} \kappa \Delta U \, \d t  + \sqrt{C_d \kappa} \sum_{k,i} \frac{1}{|k|^{\gamma}} \Pi(\sigma_{k,i}\cdot \nabla \tilde{u} )\, \d W_{t}^{k,i},
\end{equation}
where $V= K(U)$. We will first show the well-posedness of \eqref{CLT-limit-eq} and then prove the convergence of $U^N$ to $U$. In this section we assume that
  \begin{equation} \label{condition-gamma-0}
  \frac{d-2}{4}<\gamma_0<\frac{d+2}{4} \quad \mbox{and} \quad 2\gamma_0+\gamma>\frac{d}{2}.
  \end{equation}

\subsection{Well-posedness of the limit equation} \label{subs-limit-eq}

This subsection deals with the well-posedness of  the limit equation \eqref{CLT-limit-eq}, which is also interpreted in mild form:
\begin{equation}\label{eq:CLT-limit-eq-mild}
  U_t = - \int_0^t P_{t-s} \big[ \Pi \big( V_s \cdot\nabla \tilde{u}_s + \tilde{v}_s \cdot\nabla U_s \big) \big]\,\d s +\tilde{Z}_t,
\end{equation}
where $P_t= e^{C_d^{\prime} \kappa t\Delta}\, (t\geq 0)$ and the stochastic convolution $\tilde{Z}_t$ is
\begin{equation} \label{CLT-Z_t}
  \tilde{Z}_t:= \sqrt{C_d \kappa} \sum_{k,i} \frac{1}{|k|^{\gamma}} \int_{0}^{t} P_{t-s} \bigl[ \Pi \big(\sigma_{k,i} \cdot \nabla \tilde{u}_s \big) \bigr] \d W_{s}^{k,i}.
\end{equation}
First we study the regularity of the stochastic convolution $\{\tilde Z_t\}_{t\ge 0}$.

\begin{lemma}\label{lem:stoch-convol}
  For any $ \beta >\frac{d}{2}-\gamma$ and $p\in [1,\infty)$, it holds
  $ \E \big[ \|\tilde{Z} \|_{C^0 H^{-\beta}}^p \big] < +\infty$.
\end{lemma}

\begin{proof}
It is enough to consider the case $\beta <\frac{d}{2}$.
Fix $0<\eps < \beta+\gamma-\frac{d}{2}$; by Burkholder-Davis-Gundy's inequality, we have
\begin{align*}
\E\big[ \|\tilde{Z}_t \|_{H^{-\beta}}^{2p} \big]
&= |C_d \kappa|^p \E\bigg[ \Big\| \sum_{k,i} |k|^{-\gamma} \int_0^t P_{t-s} \Pi (\sigma_{k,i} \cdot\nabla \tilde{u}_s)\, \d W^{k,i}_s \Big\|_{H^{-\beta}}^{2p} \bigg] \\
&\lesssim \E\bigg[ \Big( \sum_{k,i} |k|^{-2\gamma} \int_0^t \| P_{t-s} \Pi (\sigma_{k,i} \cdot\nabla \tilde{u}_s) \|_{H^{-\beta}}^2 \,\d s \Big)^{p} \bigg] \\
&\lesssim \E\bigg[ \Big( \sum_{k,i} |k|^{-2\gamma} \int_0^t (t-s)^{\eps-1} \| \sigma_{k,i} \cdot\nabla \tilde{u}_s \|_{H^{-\beta-1+\epsilon}}^2 \,\d s \Big)^{p} \bigg],
\end{align*}
where in the last step we used the properties of semigroup.
Similar to \eqref{norm-transport-term}, we have
  \[ \sum_{k,i} |k|^{-2\gamma} \| \sigma_{k,i} \cdot\nabla \tilde{u}_s \|_{H^{-\beta-1+\eps}}^2
  \lesssim \sum_k |k|^{-2\gamma} \|e_k \, \tilde{u}_s\|_{H^{-\beta+\eps}} \le \sum_{k,l \in \Z_0^d} |k|^{-2\gamma} |l|^{-2(\beta-\eps)} |\<\tilde{u}_s, e_{k-l}\>|^2.\]
Applying Lemma \ref{cor-convolution sum} with $a=\gamma$, $b=\beta-\eps$ and $c=0$ we have
\[ \sum_{k,i} |k|^{-2\gamma} \| \sigma_{k,i} \cdot \nabla \tilde{u}_s \|_{H^{-\beta-1+\eps}}^2 \lesssim_{\beta, \gamma, \eps} \|\tilde{u}_s \|_{L^2}^2. \]
Substituting this estimate into the above inequality yields
\[ \E\big[ \|\tilde{Z}_t \|_{H^{-\beta}}^{2p} \big]^{1/2p}
\lesssim \E\bigg[ \Big( \int_0^t (t-s)^{\eps-1} \|\tilde{u}_s \|_{L^2}^2 \,\d s \Big)^{p} \bigg]^{1/2p} \lesssim_{\eps} t^{\eps/2} \|u_0 \|_{L^2}. \]
A similar computation gives us
\[ \E\bigg[ \Big\| \sum_{k,i} \frac{1}{|k|^{\gamma}} \int_{s}^{t} P_{t-r} \bigl[ \Pi \big( \sigma_{k,i} \cdot \nabla \tilde{u}_r \big) \bigr] \d W_{r}^{k,i} \Big\|_{H^{-\beta}}^{2p} \bigg]^{1/2p} \lesssim |t-s|^{\eps/2} \|u_0 \|_{L^2}.\]

Next, observing that by definition $\tilde{Z}$ satisfies
$$\tilde{Z}_t = P_{t-s} \tilde{Z}_s + \sum_{k,i} \frac{1}{|k|^{\gamma}} \int_{s}^{t} P_{t-r} \bigl[ \Pi \big( \sigma_{k,i} \cdot \nabla \tilde{u}_r \big) \bigr] \d W_{r}^{k,i} ; $$
and thus
  $$\tilde{Z}_t-\tilde{Z}_s = (P_{t-s} -I) \tilde Z_s + \sum_{k,i} \frac{1}{|k|^{\gamma}} \int_{s}^{t} P_{t-r} \bigl[ \Pi \big( \sigma_{k,i} \cdot \nabla \tilde{u}_r \big) \bigr] \d W_{r}^{k,i}. $$
By Lemma \ref{lem-semigroup}(ii), we obtain
\begin{align*}
\|\tilde{Z}_t-\tilde{Z}_s\|_{H^{-\beta-\eps}}
& \leq \| (I- P_{t-s}) \tilde{Z}_s\|_{H^{-\beta-\eps}} + \Big\| \sum_{k,i} \frac{1}{|k|^{\gamma}} \int_{s}^{t} P_{t-r} \bigl[ \Pi \big( \sigma_{k,i} \cdot \nabla \tilde{u}_r \big) \bigr] \d W_{r}^{k,i} \Big\|_{H^{-\beta-\eps}} \\
& \lesssim |t-s|^{\eps/2} \| \tilde{Z}_s\|_{H^{-\beta}} + \Big\| \sum_{k,i} \frac{1}{|k|^{\gamma}} \int_{s}^{t} P_{t-r} \bigl[ \Pi \big( \sigma_{k,i} \cdot \nabla \tilde{u}_r \big) \bigr] \d W_{r}^{k,i} \Big\|_{H^{-\beta}}.
\end{align*}
Taking expectation and applying the previous estimates we arrive at
\begin{align*}
\E\big[ \| \tilde{Z}_t-\tilde{Z}_s\|_{H^{-\beta-\eps}}^{2p}\big]^{1/2p}
& \lesssim_{p,T} |t-s|^{\eps/2}  \|u_0 \|_{L^2} .
\end{align*}
Choosing $p > 1/\eps$ and applying Kolmogorov's continuity criterion, we obtain the desired result, up to renaming $\beta -\eps$ as $\beta$.
\end{proof}

Inspired by the mild formulation \eqref{eq:CLT-limit-eq-mild} and Lemma \ref{lem:stoch-convol}, given a trajectory $z\in C^0_t H^{-\beta}_x$, we turn to the study of the \textit{analytic} equation
\begin{equation}\label{eq:CLT-limit-analytic}
y_t = - \int_0^t P_{t-s} \Pi \big[ K(y_s)\cdot\nabla \tilde{u}_s + K(\tilde{u}_s) \cdot \nabla y_s \big]\,\d s +z_t.
\end{equation}
Recall that $K=(1+(-\Delta)^{\gamma_0} )^{-1}$ is a linear operator.

\begin{proposition}\label{prop:CLT-limit-analytic}
  For $\frac{d}{2}-\gamma<\beta< (2\gamma_0) \wedge \frac{d}{2}$ and any $z\in C^0_t H^{-\beta}_{x}$, there exists a unique solution $y\in C^0_t H^{-\beta}_x$ to \eqref{eq:CLT-limit-analytic}.
  Moreover, the solution map $z\mapsto y=:Sz$ is a linear bounded operator and there exists $C>0$ and $q>\max \{ 2, \frac{4}{4\gamma_0-d+2} \}$ such that
  \begin{equation}\label{eq:CLT-limit-estim}
  \| S z\|_{C^0 H^{-\beta}} \lesssim \exp\big(C(1+T)\| u_0\|_{L^2}^{q} \big)\, \| z \|_{C^0 H^{-\beta}}.
  \end{equation}
\end{proposition}

\begin{proof}
  Let us define $a(t):=\int_0^t \| \tilde{u}_s\|_{H^1}^2 \, \d s+ t \|u_0\|_{L^2}^{\eta} $, the constant $\eta$ is to be determined; and endow $C^0_t H^{-\beta}_x$ with the equivalent norm
  \begin{equation*}
  \| y\|_\lambda := \sup_{t\in [0,T]} \big\{ e^{-\lambda a(t)} \| y_t\|_{H^{-\beta}} \big\}
  \end{equation*}
  for a suitable $\lambda>0$ to be chosen later. Define a map $\Gamma$ on $C^0_t H^{-\beta}_{x}$ by
  \begin{equation*}
  (\Gamma y)_t := - \int_0^t P_{t-s} \Pi \big[ K( y_s)\cdot\nabla \tilde{u}_s + K(\tilde{u}_s)\cdot\nabla y_s \big]\,\d s +z_t.
  \end{equation*}
  We are going to show that $\Gamma$ is a contraction on $(C^0_t H^{-\beta}_x, \| \cdot\|_\lambda)$ for some large $\lambda>0$, which immediately implies existence and uniqueness of solutions to \eqref{eq:CLT-limit-analytic}; since $\Gamma$ is an affine map, the same computation shows that indeed $\Gamma y\in C^0_t H^{-\beta}_x$ whenever $y$ does so.

  Given $y^1,\, y^2\in C^0_t H^{-\beta}_x$, set $y^0=y^1-y^2$, then by Lemma \ref{lem-semigroup}(i) and (iii) it holds
  \begin{equation}\label{prop:CLT-limit-analytic.1}
  \aligned
  & \| (\Gamma y^1-\Gamma y^2)_t\|_{H^{-\beta}}^2 \\
  &=\bigg\| \int_0^t P_{t-s} \Pi \big[ K(y^0_s)\cdot \nabla \tilde{u}_s+ K(\tilde{u}_s)\cdot \nabla y^0_s \big] \,\d s\bigg\|_{H^{-\beta}}^2\\
  & \lesssim \Big(\int_0^t \big\| P_{t-s} \Pi \big[ K(y^0_s) \cdot \nabla \tilde{u}_s \big] \big\|_{H^{-\beta}} \,\d s \Big)^{2} + \int_0^t \| K(\tilde{u}_s)\cdot \nabla y^0_s\|_{H^{-\beta-1}}^2 \,\d s \\
  & \lesssim \Big( \int_0^t \frac{1}{(t-s)^{1-\eps}} \| K(y^0_s)\cdot \nabla \tilde{u}_s\|_{H^{-\beta-2+2\eps}} \,\d s \Big)^{2} + \int_0^t \| K(\tilde{u}_s)\cdot \nabla y^0_s\|_{H^{-\beta-1}}^2 \,\d s,
  \endaligned
  \end{equation}
  for some $\eps = \min\{\frac{1}{2}, \gamma_0 -\frac{d-2}{4}\} >\frac{1}{q}$.
  We denote the last terms on right-hand side by $I^1_t$ and $I^2_t$, respectively.

First, applying Lemma \ref{lem-nD-nonlinearity}-(ii) with $b=\beta+1-2\eps \in (0,\frac{d}{2})$, and noting that $\frac{d}{2} -b=2\eps+\frac{d-2}{2}-\beta \leq 2 \gamma_0 -\beta$, by Poincare inequality we obtain
  \begin{align*}
  I^1_t
  &\lesssim \Big( \int_0^t \frac{1}{(t-s)^{1-\eps}} \| K (y^0_s) \|_{H^{2 \gamma_0-\beta}} \| \tilde{u}_s \|_{L^2} \, \d s \Big)^2
  \lesssim \Big( \int_0^t \frac{1}{(t-s)^{1-\eps}} \| y^0_s \|_{H^{-\beta}} \| \tilde{u}_s \|_{L^2} \, \d s \Big)^2 \\
  &\lesssim  e^{2\lambda a(t)} \|y^0\|_\lambda^2 \|u_0\|_{L^2}^2 \Big( \int_0^t \frac{1}{(t-s)^{1-\eps}} e^{-\lambda (a(t)-a(s))} \, \d s \Big)^2 .
  \end{align*}
Note that $a(t)-a(s) \geq \| u_0 \|_{L^2}^{\eta} (t-s)$; letting $r= \lambda \| u_0\|^{\eta} (t-s)$, we have
  $$ \aligned
  I^1_t & \lesssim e^{2\lambda a(t)} \| y^0\|_\lambda^2 \|u_0\|_{L^2}^2 \Big( \int_0^t \frac{1}{(t-s)^{1-\eps}} e^{-\lambda \| u_0 \|_{L^2}^{\eta} (t-s)} \, \d s \Big)^2 \\
  & \leq e^{2\lambda a(t)} \| y^0\|_\lambda^2 \|u_0\|_{L^2}^2 \big(\lambda \| u_0 \|_{L^2}^{\eta} \big)^{-2\eps} \Big( \int_0^{\infty} r^{\eps-1} e^{-r} \, \d r \Big)^2 .
  \endaligned $$
  Choosing $\eta= \frac{1}{2 \eps}$, we obtain $I^1_t \lesssim e^{2\lambda a(t)} \| y^0\|_\lambda^2 \|u_0\|_{L^2} \lambda^{-2\eps} $.

Next, applying Lemma \ref{lem-nD-nonlinearity}-(i) with $a=1+2\gamma_0>\frac{d}{2},\, b=\beta < \frac{d}{2}$ leads to
  \begin{align*}
  I^2_t
  &\lesssim \int_0^t \| K(\tilde{u}_s)\|_{H^{1+2\gamma_0}}^2 \| y^0_s \|_{H^{-\beta}}^2 \, \d s
  \lesssim \| y^0\|_\lambda^2 \int_0^t e^{2\lambda a(s)} \| \tilde{u}_s\|_{H^1}^2 \, \d s
  \lesssim \frac{\| y^0\|_\lambda^2}{2\lambda}\,  e^{2\lambda a(t)}.
  \end{align*}
Combining the above estimates, multiplying both sides of \eqref{prop:CLT-limit-analytic.1} by $e^{-2\lambda a(t)}$ and taking the supremum over $t\in [0,T]$, we obtain the existence of a constant $C>0$, independent of $\lambda$, such that
  \begin{equation*}
  \|\Gamma y^1- \Gamma y^2 \|_\lambda^2
  \leq C \max\Big\{ \frac{1}{\lambda}, \frac{1}{\lambda^{2\eps}} \Big\} \| y^0\|_\lambda^2
  = C \max\Big\{ \frac{1}{\lambda}, \frac{1}{\lambda^{2\eps}} \Big\} \, \| y^1- y^2 \|_\lambda^2;
  \end{equation*}
contractility of the mapping $\Gamma$ follows by choosing $\lambda$ large enough.

Due to the affine structure of \eqref{eq:CLT-limit-analytic}, $z\mapsto Sz$ is a linear operator. To show the boundedness of this operator, we only need to prove \eqref{eq:CLT-limit-estim}. If $y$ solves \eqref{eq:CLT-limit-analytic}, then for any $t\in [0,T]$ and $q> \max\{2,\frac{4}{4\gamma_0-d+2}\}$, similar to the above estimates, we have
  \begin{align*}
  \| y_t \|_{H^{-\beta}}^q
  & \lesssim \bigg\| \int_0^t P_{t-s} \Pi \big[K(y_s) \cdot \nabla \tilde{u}_s+ K(\tilde{u}_s)\cdot \nabla y_s \big] \,\d s \bigg\|_{H^{-\beta}}^q + \| z_t\|_{H^{-\beta}}^q\\
  & \lesssim \Big( \int_0^t \frac{1}{(t-s)^{1-\eps}} \| y_s \|_{H^{-\beta}} \| \tilde{u}_s \|_{L^2} \, \d s \Big)^q + \Big(\int_0^t \| \tilde{u}_s\|_{H^1}^2 \, \| y_s\|_{H^{-\beta}}^2 \, \d s\Big)^{\frac{q}{2}} + \| z\|_{C^0 H^{-\beta}}^q.
  \end{align*}
  Similar to the proof of Lemma \ref{lem-main-nonlin},
  by H\"older inequality we have
  $$ \| y_t \|_{H^{-\beta}}^q \lesssim \int_0^t \bigg(\|u_0 \|_{L^2}^q +  \Big( \int_0^t \| \tilde{u}_r\|_{H^1}^2 \, \d r \Big)^{\frac{q-2}{2}}  \|\tilde u_s \|_{H^1}^2 \bigg) \|y_s\|_{H^{-\beta}}^q\,\d s + \| z\|_{C^0 H^{-\beta}}^q.$$
  Applying Gronwall's lemma and using the fact $\| \tilde{u}_t\|_{L^2}^2 + C_{d}^{\prime} \kappa \int_0^t \| \nabla \tilde{u}_s\|_{L^2}^2\, \d s \leq \| u_0\|_{L^2}^2$, we get the estimate \eqref{eq:CLT-limit-estim}.
\end{proof}

\begin{corollary}\label{cor:CLT-limit-eq}
  Let $u_0\in H_\sigma$ be fixed, then there exists a unique solution $U$ to \eqref{eq:CLT-limit-eq-mild}, which is given by $U=S \tilde{Z}$. In particular, $U$ is a Gaussian field which satisfies
  $\E\big[\| U \|_{C^0 H^{-\beta}}^q \big]<\infty$ for any $q>\max \{ 2, \frac{4}{4\gamma_0-d+2} \}$ and $\beta> \frac{d}{2}-\gamma$.
\end{corollary}

\begin{proof}
  Due to Proposition \ref{prop:CLT-limit-analytic}, we can get the strong existence and pathwise uniqueness of the solution to \eqref{eq:CLT-limit-eq-mild}. By Lemma \ref{lem:stoch-convol} and estimate \eqref{eq:CLT-limit-estim}, we know that $\E\big[\| U \|_{C^0 H^{-\beta}}^q \big]<\infty$ for any $q>\max \{ 2, \frac{4}{4\gamma_0-d+2} \}$.
  Because the operator $z \mapsto Sz$ is linear and $\tilde{Z}$ is a Gaussian field, the solution $U= S\tilde{Z}$ is Gaussian.
\end{proof}

\subsection{Proof of CLT} \label{subs-proof-CLT}

In this subsection, we estimate the convergence speed of $U^N$ to $U$ in the case $d=3$. The first condition in \eqref{condition-gamma-0} now reads as $\frac14< \gamma_0 < \frac54$, and we want to keep $\gamma_0$ small; due to the constraint $2\gamma_0 + \gamma>\frac{3}{2}$, we choose $\gamma\in (1,\frac32)$ in the sequel. An added benefit of the condition $\gamma>1$ is that we do not need to deal with different expressions of $D_N$, cf. \eqref{def-D_N}.

\begin{proof}[Proof of Theorem \ref{thm:CLT}]
As before, writing the equation \eqref{CLT-approx-eq} in mild form yields
  \begin{equation} \label{CLT-N}
  \aligned
  U^N_t & =- \int_0^t P_{t-s} \Pi \big[ K(U^N_s) \cdot \nabla u^N_s + K(\tilde{u}_s) \cdot \nabla U^N_s\big]\, \d s \\
  & \quad + \frac{1}{\sqrt{\epsilon_N}} \int_0^t P_{t-s} \Big(S_{\theta^N}^{(3)} (u^N_s)- C_3^{\prime} \kappa \Delta u^N_s\Big)\,\d s + \tilde{Z}^N_t,
  \endaligned
  \end{equation}
where the stochastic convolution $\tilde{Z}_t^N$ is
  \[ \tilde{Z}^N_t:= \sqrt{C_3 \kappa} \sum_{|k| \leq N} \sum_{i=1}^{2} \frac{1}{|k|^{\gamma}} \int_{0}^{t} P_{t-s} \bigl[ \Pi \big(\sigma_{k,i} \cdot \nabla u^N_s \big) \bigr] \d W_{s}^{k,i} .\]
Recall also the mild formulation \eqref{eq:CLT-limit-eq-mild} of the limit equation with the stochastic convolution $\tilde{Z}_t$ defined in \eqref{CLT-Z_t}. Therefore, for any $t\in [0,T]$ and $q>\max \{2,\frac{4}{4\gamma_0-1}\}$, we have
  \begin{align*}
  \E\big[ \|U^N_t -U_t \|_{H^{-\alpha_0}}^q \big] \lesssim I^N_1(t) + I^N_2(t) + I^N_3(t) + I^N_4(t)+I^N_5(t),
  \end{align*}
  where
  \begin{align*}
  I^N_1(t) & = \E\bigg[ \Big\| \int_0^t P_{t-s} \Pi \big[ K(U^N_s -U_s) \cdot \nabla u^N_s \big]\, \d s \Big\|_{H^{-\alpha_0}}^q\bigg], \\
  I^N_2(t) & = \E\bigg[ \Big\| \int_0^t P_{t-s} \Pi \big[ K (\tilde{u}_s) \cdot\nabla (U^N_s -U_s) \big]\, \d s \Big\|_{H^{-\alpha_0}}^q \bigg], \\
  I^N_3(t) & = \E\bigg[ \Big\| \int_0^t P_{t-s} \Pi \big[ K(U_s) \cdot\nabla (u^N_s -\tilde{u}_s) \big]\, \d s \Big\|_{H^{-\alpha_0}}^q \bigg], \\
  I^N_4(t) & = \frac{1}{\epsilon_N^{q/2}} \E\bigg[ \Big\| \int_0^t P_{t-s} \Big(S_{\theta^N}^{(3)} (u^N_s)- C_3^{\prime} \kappa \Delta u^N_s\Big)\,\d s \Big\|_{H^{-\alpha_0}}^q \bigg], \\
  I^N_5(t) & =\E \big[\| \tilde{Z}^N_t-\tilde{Z}_t \|_{H^{-\alpha_0}}^q \big].
  \end{align*}

\textbf{Step 1: preliminary estimates.}
  Set $\eps_1 = \min \{\frac{1}{2},\gamma_0 -\frac{1}{4}\}$; then by Lemma \ref{lem-semigroup}(i) it holds
  \begin{align*}
  I^N_1(t)
  &\lesssim \E\bigg[ \Big( \int_0^t \big\| P_{t-s} \Pi \big[K(U^N_s -U_s) \cdot \nabla u^N_s \big] \big\|_{H^{-\alpha_0}} \,\d s \Big)^q \bigg] \\
  &\lesssim \E\bigg[ \Big( \int_0^t \frac{1}{(t-s)^{1-\eps_1}} \| K(U^N_s-U_s) \cdot \nabla u^N_s \|_{H^{-\alpha_0-2+2\eps_1}} \,\d s \Big)^{q} \bigg].
  \end{align*}
Using Lemma \ref{lem-nD-nonlinearity}-(ii) with $b=\alpha_0+1-2\eps_1 \in (0,\frac{3}{2})$, and noting that $\frac32 -b=2\eps_1+\frac{1}{2}-\alpha_0 \leq 2\gamma_0 -\alpha_0$, by Poincare inequality we obtain
  \begin{align*}
  I^N_1(t)
  &\lesssim \E\bigg[ \Big( \int_0^t \frac{1}{(t-s)^{1-\eps_1}} \| K(U^N_s-U_s) \|_{H^{2\gamma_0-\alpha_0}} \| u^N_s \|_{L^2} \,\d s \Big)^{q} \bigg] \\
  &\lesssim \E\bigg[ \Big( \int_0^t \frac{1}{(t-s)^{1-\eps_1}} \| U^N_s-U_s \|_{H^{-\alpha_0}} \| u^N_s \|_{L^2} \,\d s \Big)^{q} \bigg],
  \end{align*}
Since $q>\max \{2,\frac{4}{4\gamma_0-1}\}$, one has $(1-\eps_1) q' <1$; H\"older's inequality yields
$$ I^N_1(t) \lesssim \| u_0\|_{L^2}^{q} \int_0^t \E \bigl[ \| (U^N_s-U_s) \|_{H^{-\alpha_0}}^q \bigr] \, \d s .$$

The estimate of $I^N_2(t)$ is similar to $I_{1,2,N}(t)$ in the proof of Lemma \ref{lem-main-nonlin}. We use the semigroup property and Lemma \ref{lem-nD-nonlinearity}-(i) with $a=1+2\gamma_0>\frac{3}{2}$ and $b=\alpha_0<\frac{3}{2}$, which gives
\begin{align*}
I^N_2(t) \lesssim \E\bigg[ \Big( \int_0^t \| K(\tilde{u}_s) \cdot \nabla (U^N_s-U_s)\|_{H^{-\alpha_0-1}}^2 \,\d s \Big)^{\frac{q}{2}} \bigg]
\lesssim \E\bigg[ \Big( \int_0^t \| \tilde{u}_s\|_{H^1}^2 \|U^N_s-U_s\|_{H^{-\alpha_0}}^2 \, \d s \Big)^{\frac{q}{2}} \bigg] .
\end{align*}
Choosing $q> 2$, then by H\"older inequality we have
$$ I^N_2(t) \lesssim \| u_0\|_{L^2}^{q-2} \E\bigg[  \int_0^t \| \tilde{u}_s\|_{H^1}^2 \|U^N_s-U_s\|_{H^{-\alpha_0}}^q \, \d s \bigg]. $$
Combining the above estimates yields
\begin{align*}
\E \big[\| U^N_t-U_t\|_{H^{-\alpha_0}}^q \big]
\lesssim & \int_0^t \big(\| u_0\|_{L^2}^q + \| u_0\|_{L^2}^{q-2} \|\tilde{u}_s\|_{H^1}^2 \big) \,\E\big[\| U^N_s-U_s\|_{H^{-\alpha_0}}^q \big]\, \d s \\
& \, + I^N_3(t) + I^N_4(t) +I^N_5(t) ,
\end{align*}
and so by Gronwall's lemma we find
\begin{equation}\label{eq:CLT-proof.1}
\sup_{t\in [0,T]} \E \big[ \| U^N_t-U_t\|_{H^{-\alpha_0}}^q \big]
\lesssim e^{C(1+T)\| u_0\|_{L^2}^q} \sup_{t\in [0,T]} \big(  I^N_3(t) + I^N_4(t) +I^N_5(t)\big).
\end{equation}

\textbf{Step 2: estimate of $I^N_3(t)$.} By the property of semigroup, it holds
  \begin{align*}
    I^N_3(t)
    &\lesssim \E\bigg[ \Big( \int_0^t \big\| P_{t-s} \Pi [K(U_s) \cdot \nabla (u^N_s-\tilde{u}_s)] \big\|_{H^{-\alpha_0}} \,\d s \Big)^q \bigg] \\
    &\lesssim \E\bigg[ \Big( \int_0^t \frac{1}{(t-s)^{1-\eps_2}} \| K(U_s) \cdot \nabla (u^N_s-\tilde{u}_s) \|_{H^{-\alpha_0-2+2\eps_2}} \,\d s \Big)^{q} \bigg],
  \end{align*}
where $\eps_2>0$. Recalling that $\alpha_0 > \frac12$, we can choose $\eps_2$ small enough so that $\alpha_0 -\frac12 -2\eps_2>0$. Let $c=\alpha_0+\gamma-\frac{3}{2}$, by Lemma \ref{lem-nD-nonlinearity}-(iii) with $b=c>\alpha_0-\frac{1}{2}>0$, we have
  \begin{align*}
    I^N_3(t)
    &\lesssim \E\bigg[ \Big( \int_0^t \frac{1}{(t-s)^{1-\eps_2}} \| K(U_s) \|_{H^c} \| u^N_s-\tilde{u}_s \|_{H^{-c}} \,\d s \Big)^{q} \bigg] \\
    &\lesssim  \E\bigg[ \| U \|_{C^0 H^{c-2\gamma_0}}^{q} \| u^N -\tilde{u}\|_{C^0 H^{-c}}^{q} \Big( \int_0^t \frac{1}{(t-s)^{1-\eps_2}} \,\d s \Big)^{q} \bigg] \\
    &\lesssim T^{\eps_2 q}\, \E\big[ \| U \|_{C^0 H^{c-2\gamma_0}}^{2q} \big]^{1/2}\, \E\big[ \| u^N -\tilde{u}\|_{C^0 H^{-c}}^{2q} \big]^{1/2},
  \end{align*}
where in the last passage we used the Cauchy inequality. Note that $2\gamma_0-c=2\gamma_0-\alpha_0 +\frac{3}{2}-\gamma>\frac{3}{2}-\gamma$; applying Corollary \ref{cor:CLT-limit-eq} with $\beta=2\gamma_0-c$, we have $\E\big[ \| U \|_{C^0 H^{c-2\gamma_0}}^{2q} \big]<+\infty$.
  As $c<\alpha_0< (2\gamma_0) \wedge 1$, applying Theorem \ref{thm-main} with $\alpha=c$ and $\delta<c$, we obtain
  $$ \sup_{t\in [0,T]} I^N_3(t) \lesssim \E\big[ \| u^N -\tilde{u}\|_{C^0 H^{-c}}^{2q} \big]^{1/2} \lesssim \| u_0\|_{L^2}^q \epsilon_N^{q(c-\delta)/2} .$$
  Choosing $\delta=\frac{2\eps}{q(3-2\gamma)}<c$ ($\eps$ being the small number in the statement of the theorem) and by the estimate of $\epsilon_N $ (see \eqref{def-epsilon}), we obtain
  \begin{equation}\label{eq:CLT-proof.2}
  \sup_{t\in [0,T]} I^N_3(t) \lesssim \| u_0\|_{L^2}^q  N^{- q \frac{3-2\gamma}{2} (\alpha_0+\gamma- \frac32)+\eps }.
  \end{equation}

\textbf{Step 3: estimate of $ I^N_4(t) $.} By Lemma \ref{lem-semigroup}, for small $\eps_3>0$ it holds
  \begin{align*}
    I^N_4(t)
    & \lesssim \frac{1}{\epsilon_N^{q/2}} \E\bigg[ \Big( \int_0^t \frac{1}{(t-s)^{1-\eps_3}} \big\| S_{\theta^N}^{(3)} (u^N_s)- C_3^{\prime} \kappa \Delta u^N_s \big\|_{H^{-\alpha_0-2+2\eps_3}} \,\d s \Big)^{q} \bigg] \\
    & \lesssim \frac{1}{\epsilon_N^{q/2}} \Big( \int_0^t \frac{1}{(t-s)^{1-\eps_3}} \,\d s \Big)^{q} \E\bigg[  \sup_{0 \leq s \leq T} \big\| S_{\theta^N}^{(3)} (u^N_s)- C_3^{\prime} \kappa \Delta u^N_s \big\|_{H^{-\alpha_0-2+2\eps_3}}^q \bigg] \\
    & \lesssim \E\bigg[ \sup_{0 \leq s \leq T} \| u_s^N\|_{L^2}^q \bigg] \Big(\frac{1}{\sqrt{\epsilon_N}} D_N^{\alpha_0-2\eps_3}\Big)^q,
  \end{align*}
where in the last step we have applied Theorem \ref{thm-Ito-corrector} with $b=0$ and $\alpha=\alpha_0-2\eps_3<1$.
Recalling $D_N$ defined in \eqref{def-D_N}, for $1<\gamma<\frac{3}{2}$ we have $D_N \lesssim \epsilon_N$. Choosing $\eps_3 = \frac{\eps}{2(3-2\gamma)q}$ and by the fact $\sup_{0\leq s \leq T} \|u_s^N\|_{L^2} \leq \|u_0\|_{L^2}$ $\mathbb{P}$-a.s., we obtain
  \begin{equation} \label{eq:CLT-proof.3}
    \sup_{t\in [0,T]} I^N_4(t) \lesssim \|u_0\|_{L^2}^q \epsilon_N^{q(\alpha_0-2\eps_3- \frac12)} \lesssim \| u_0\|_{L^2}^q N^{-q(3-2\gamma)(\alpha_0-\frac{1}{2})+\eps} .
  \end{equation}

\textbf{Step 4: estimate of $I^N_5(t)$.} We split it in two parts: $I^N_5(t)\lesssim J^N_1(t) + J^N_2(t)$ where
  \begin{align*}
  & J^N_1(t):=\E\bigg[\Big\| \sqrt{C_3 \kappa} \sum_{|k| > N} \sum_{i=1}^{2} \frac{1}{|k|^{\gamma}} \int_{0}^{t} P_{t-s} \Pi \bigl[ \sigma_{k,i} \cdot \nabla \tilde{u}_s \bigr] \d W_{s}^{k,i} \Big\|_{H^{-\alpha_0}}^q \bigg], \\
  & J^N_2(t):= \E\bigg[ \Big\| \sqrt{C_3 \kappa} \sum_{|k| \leq N} \sum_{i=1}^{2} \frac{1}{|k|^{\gamma}} \int_{0}^{t} P_{t-s} \Pi \bigl[ \sigma_{k,i} \cdot \nabla \big( u^N_s -\tilde{u}_s\big) \bigr] \d W_{s}^{k,i} \Big\|_{H^{-\alpha_0}}^q \bigg].
  \end{align*}
  By Burkholder-Davis-Gundy's inequality, we have
  \begin{align*}
  J^N_1(t)
  &\lesssim \E\bigg[ \Bigl( \sum_{|k|>N} \sum_{i=1}^{2} |k|^{-2\gamma} \int_0^t \| P_{t-s} \Pi (\sigma_{k,i} \cdot\nabla \tilde{u}_s)\|_{H^{-\alpha_0}}^2\, \d s \Bigr)^{\frac{q}{2}} \bigg] \\
  &\lesssim \E\bigg[ \Bigl( \sum_{|k|>N} \sum_{i=1}^{2} |k|^{-2\gamma} \int_0^t \frac{1}{(t-s)^{1-\eps_4}} \|\sigma_{k,i} \cdot\nabla \tilde{u}_s\|_{H^{-\alpha_0-1+\eps_4}}^2\, \d s \Bigr)^{\frac{q}{2}} \bigg],
  \end{align*}
  where $\eps_4 \in (0,\alpha_0-1/2)$. Noting that $\sigma_{k,i} = a_{k,i} e_k$ is divergence free, for any $r \in \R$ it holds
  \begin{equation}\label{eq:basic-estim-CLT-proof}
  \| \sigma_{k,i} \cdot\nabla \tilde{u}_s\|_{H^{-r-1}}
  = \| \nabla\cdot (\sigma_{k,i} \otimes \tilde{u}_s)\|_{H^{-r-1}}
  \lesssim \| \sigma_{k,i} \otimes \tilde{u}_s \|_{H^{-r}} \lesssim
  \| e_k\, \tilde{u}_s\|_{H^{-r}}.
  \end{equation}
  Letting $r=\alpha_0-\eps_4$, we have
  \begin{align*}
  J^N_1(t)
  &\lesssim \E\bigg[ \Bigl( \sum_{|k|>N} |k|^{-2\gamma} \int_0^t \frac{1}{(t-s)^{1-\eps_4}} \| e_k\, \tilde{u}_s \|_{H^{-\alpha_0+\eps_4}}^2\, \d s \Bigr)^{\frac{q}{2}} \bigg] \\
  &\lesssim \E \bigg[ \Bigl( \sup_{s \in [0,T]} \sum_{|k|>N} |k|^{-2\gamma} \| e_k\, \tilde{u}_s\|_{H^{-\alpha_0+\eps_4}}^2 \Bigr)^{\frac{q}{2}} \bigg].
  \end{align*}
  Letting $\delta_2=\alpha_0-\frac{1}{2}-\eps_4$, then $0<\delta_2<\gamma$ and it holds
  \begin{align*}
  \sum_{|k|>N} |k|^{-2\gamma} \| e_k\, \tilde{u}_s\|_{H^{-\alpha_0+\eps_4}}^2
  & = \sum_{|k|>N} \sum_{l \in \Z_0^d} |k|^{-2\gamma} |2 \pi l|^{-2(\alpha_0-\eps_4)} | \< \tilde{u}_s,e_{k-l}\>|^2 \\
  & \lesssim N^{-2\delta_2} \sum_{k,l \in \Z_0^d} |k|^{-2(\gamma-\delta_2)} |l|^{-2(\alpha_0-\eps_4)} | \< \tilde{u}_s,e_{k-l}\>|^2.
  \end{align*}
As $\gamma-\delta_2+\alpha_0-\eps_4=\gamma+\frac{1}{2}> \frac{3}{2}$, applying Lemma \ref{cor-convolution sum} with $a=\gamma- \delta_2$, $b=\alpha_0-\eps_4$ and $c=0$, we obtain
  $$ J^N_1(t) \lesssim \E \bigg[ \Bigl( \sup_{s \in [0,T]} \sum_{|k|>N} |k|^{-2\gamma} \| e_k\, \tilde{u}_s\|_{H^{-\alpha_0+\eps_4}}^2 \Bigr)^{\frac{q}{2}} \bigg] \lesssim N^{-q \delta_2} \E \bigg[ \sup_{s \in [0,T]} \|\tilde{u}_s\|_{L^2}^{q} \bigg] .$$
Choosing $\eps_4< \frac{\eps}{q}$ (remark that $\eps$ is the small number in the statement of Theorem \ref{thm:CLT}) and by the fact $\sup_{0\leq s \leq T} \|\tilde{u}_s\|_{L^2} \leq \|u_0\|_{L^2}$, we obtain
  \begin{equation}\label{eq:CLT-proof.4}
  \sup_{t\in [0,T]} J^N_1(t) \lesssim N^{-q (\alpha_0- \frac12)+\eps}.
  \end{equation}

We turn to estimating $J^N_2$. By Burkholder-Davis-Gundy's inequality, we have
  \begin{align*}
  J^N_2(t)
  & \lesssim \E\bigg[ \Bigl( \sum_{|k| \leq N} \sum_{i=1}^{2} |k|^{-2\gamma} \int_0^t \big\| P_{t-s} (\sigma_{k,i} \cdot\nabla \big( u^N_s -\tilde{u}_s\big) \big\|_{H^{-\alpha_0}}^2\, \d s \Bigr)^{\frac{q}{2}} \bigg] \\
  &\lesssim \E\bigg[ \Bigl( \sum_{|k|\leq N} \sum_{i=1}^{2} |k|^{-2\gamma} \int_0^t \frac{1}{(t-s)^{1-\eps_5}} \big\|\sigma_{k,i} \cdot\nabla \big( u^N_s -\tilde{u}_s\big) \big\|_{H^{-\alpha_0-1+\eps_5}}^2\, \d s \Bigr)^{\frac{q}{2}} \bigg],
  \end{align*}
where $\eps_5 \in (0,\alpha_0-1/2)$. Similar to the estimate \eqref{eq:basic-estim-CLT-proof}, for $r=\alpha_0-\eps_5$ we have
  \begin{align*}
  J^N_2(t)
  & \lesssim \E\bigg[ \Bigl( \sum_{|k| \leq N} |k|^{-2\gamma} \int_0^t \frac{1}{(t-s)^{1-\eps_5}} \big\| e_k\, \big( u^N_s -\tilde{u}_s \big) \big\|_{H^{-\alpha_0+\eps_5}}^2\, \d s \Bigr)^{\frac{q}{2}} \bigg] \\
  & \lesssim \E \bigg[ \Bigl( \sup_{s \in [0,T]} \sum_{k} |k|^{-2\gamma} \big\| e_k\, \big( u^N_s -\tilde{u}_s \big) \big\|_{H^{-\alpha_0+\eps_5}}^2 \Bigr)^{\frac{q}{2}} \bigg] \\
  & = \E \bigg[ \Bigl( \sup_{s \in [0,T]} \sum_{k,l \in \Z_0^d} |k|^{-2\gamma} |2 \pi l|^{-2(\alpha_0-\eps_5)} |\< u^N_s -\tilde{u}_s,e_{k-l} \>|^2 \Bigr)^{\frac{q}{2}} \bigg].
  \end{align*}
Let $\delta_3=\alpha_0-\frac{1}{2}-\eps_5<\gamma$. As $\gamma+\alpha_0-\eps_{5}-\delta_3=\gamma+\frac{1}{2}>\frac{3}{2}$, applying Lemma \ref{cor-convolution sum} with $a=\gamma$, $b=\alpha_0-\eps_{5}$ and $c=\delta_3$, we get
  $$J^N_2(t) \lesssim \E \bigl[ \| u^N-\tilde{u}\|_{C^0 H^{-\delta_3}}^q \bigr].$$
Let $\eps_5 <\frac{2 \eps}{q(3-2\gamma)}$. Applying Theorem \ref{thm-main} with $\alpha=\delta_3< \frac12 <(2\gamma_0) \wedge 1$ and $\delta = \frac{2 \eps}{q(3-2\gamma)}-\eps_5>0$, we have
  \begin{equation}\label{eq:CLT-proof.5}
    \sup_{t\in [0,T]} J^N_2(t) \lesssim \epsilon_{N}^{q(\delta_3-\delta)/2} \lesssim N^{-q \frac{3-2\gamma}{2}(\alpha_0- \frac12)+\eps}.
  \end{equation}
Due to $\frac{3-2\gamma}{2}<\frac{1}{2}<1$, summarizing estimates \eqref{eq:CLT-proof.4} and \eqref{eq:CLT-proof.5} leads to
  $$\sup_{t\in [0,T]} I^N_5(t) \lesssim N^{- q \frac{3-2\gamma}{2}(\alpha_0- \frac12) +\eps} .$$

\textbf{Step 5: final step.}  Inserting the above estimates into \eqref{eq:CLT-proof.1}, we obtain
  $$ \sup_{t\in [0,T]} \E \big[\| U^N_t-U_t\|_{H^{-\alpha_0}}^q \big] \lesssim e^{C(1+T)\| u_0\|_{L^2}^q} \max \big\{ N^{-q \frac{3-2\gamma}{2}(\alpha_0- \frac12)+\eps}, N^{- q \frac{3-2\gamma}{2}(\alpha_0+\gamma- \frac32) +\eps} \big\}.$$
Due to $\alpha_0-1/2<\alpha_0+\gamma-3/2$, we get \eqref{eq:rate-CLT-3D}.
\end{proof}

\appendix

\section{Existence of weak solutions} \label{sec-appendix-A}

In this section, we first give the definition of weak solutions to \eqref{stoch-NS-eq3} and then prove their existence by using the Galerkin approximation and the compactness method.
\begin{definition}\label{SEE-def}
  We say that \eqref{stoch-NS-eq3} has a weak solution if there exist a filtered probability space $\big( \Omega, \mathcal F, \mathcal F_t, \P\big)$, a sequence of independent $\mathcal F_t$-Brownian motions $\{W^{k,i}\}_{k\in \Z^d_0, \, i=1,\dots,d-1}$ and an $\mathcal F_t$-progressively measurable process $u \in L^2 \big(\Omega,L^2 (0,T; H) \big)$ with $\mathbb{P}$-a.s. weakly continuous trajectories such that for any $\phi \in C^\infty(\T^d;\R^d) \cap \mathcal{H}_{\sigma}$,
  \begin{equation*}
    \<u_t,\phi\> = \<u_0,\phi\> + \int_0^t \<u_s, v_s\cdot \nabla \phi\>\,\d s + \int_0^t \<u_s, S_{\theta}^{(d)}(\phi) \>\,\d s - \sqrt{C_d \kappa} \sum_{k,i} \theta_{k} \int_0^t \<u_s, \sigma_{-k,i} \cdot \nabla \phi\> \,\d W^{k,i}_s
  \end{equation*}
  holds $\P$-a.s. for all $t\in [0,T]$.
\end{definition}

The next result gives the existence of weak solutions to \eqref{stoch-NS-eq3}.

\begin{theorem}\label{thm-existence}
  For any $u_0\in H_{\sigma}$, there exists at least one weak solution to \eqref{stoch-NS-eq3} with trajectories in $L^\infty(0,T; H)$; more precisely,
  \begin{equation}\label{energy estimate thm-existence}
  \sup_{t\in [0,T]} \|u_t \|_{L^2} \leq \|u_0 \|_{L^2}\quad \mathbb{P}\mbox{-a.s.}
  \end{equation}
\end{theorem}

The proof of Theorem \ref{thm-existence} is similar to \cite[Theorem 2.2]{FGL21c}, so we only give a sketch here. We begin with introducing some notations. For $N\geq 1$, we define the finite dimensional space $ H_N=\big\{ u \in H_{\sigma} :u=\sum_{|k| \leq N } \hat{u}_k e_k  \big\} $.
Denote by $\Pi_N: H_\sigma \to H_N$ the orthogonal projection: $\Pi_N  u= \sum_{|k|\leq N} \hat{u}_k e_k$; define the linear operator $K_N :H_\sigma \to H_N$ as $K_N u= \Pi_N K(u)$. Let
\[b_N(u) = \Pi_N\big( (K_N u) \cdot \nabla(\Pi_N u) \big), \quad  G_N^{k,i}( u)= \Pi_N\big( \sigma_{k,i} \cdot \nabla(\Pi_N u) \big) \]
for all $k\in \Z^d_0$ and $i=1,2,\cdots,d-1$, and let $S_{\theta,N}^{(d)} : H_\sigma \to H_N $ as
\[ S_{\theta,N}^{(d)}(u)= C_d \kappa \sum_{k,i} \theta_{k}^2 \Pi_N \left[ \sigma_{k,i} \cdot \nabla \Pi_N \left( \sigma_{k,i} \cdot \nabla \Pi_N u \right)  \right] .\]
Note that, for fixed $N$, there are only finitely many $k \in \Z^d_0$ such that $G_N^{k,i}$ is not zero and the sum over $k$ is a finite sum. We shall view $b_N$ and $G_N^{k,i}$ as vector fields on $H_N$ which enjoy the following useful properties:
  \begin{equation}\label{properties}
  \big\< b_N( u_N),  u_N \big\> = \big\< G_N^{k,i}(u_N),  u_N \big\>=0\quad \mbox{for all }  u_N \in H_N,
  \end{equation}
which can be proved easily from the definitions of $b_N$ and $G_N^{k,i}$, and the integration by parts formula.
Consider the finite dimensional version of \eqref{stoch-NS-eq3} on $H_N$:
\begin{equation}\label{SDE}
  \left\{ \aligned
  & \d u_N(t)= -b_N( u_N(t))\,\d t + S_{\theta,N}^{(d)}(u_N(t)) \,\d t + \sqrt{C_d \kappa} \sum_{k,i} \theta_k G_N^{k,i}( u_N(t)) \, \d W^{k,i}_t, \\
  & u_N(0)= \Pi_N u_0, \quad u_0 \in H_{\sigma}.
  \endaligned \right.
\end{equation}

The following result follows directly from classical SDE theory and \eqref{properties}.

\begin{lemma}\label{energy-estimate}
The equation \eqref{SDE} has a unique strong solution $ u_N(t)$ satisfying
  $$\sup_{t\in [0,T]} \| u_N(t) \|_{L^2} \leq \| u_N(0) \|_{L^2} \quad \P \mbox{\rm-}a.s.$$
\end{lemma}
%

Next, we will show that the laws $\eta_N$ of $u_N(\cdot)$ are tight in $C\big([0,T], H_{\sigma}^{-\delta} \big)$ for any $\delta>0$ small enough. The next result about compactness follows from \cite[Corollary 9, p.90]{Simon}.

\begin{theorem} \label{thm-simon}
Let $0<\delta<\beta$ be given, if $p>12(\beta-\delta)/\delta$, then
  $$L^p(0,T; H_{\sigma}) \cap W^{1/3, 4} \big(0,T; H_{\sigma}^{-\beta} \big) \subset C\big( [0,T]; H_{\sigma}^{-\delta} \big)$$
with compact inclusion.
\end{theorem}

To show the tightness of $\{\eta_N\}_{N\geq 1}$, by Theorem \ref{thm-simon}, we need to prove $\{\eta_N\}_{N\geq 1}$ is bounded in $L^p \big(\Omega,L^p(0,T; H_{\sigma}) \big)$ and $L^4 \big(\Omega, W^{1/3,4}(0,T; H^{-\beta}_{\sigma}) \big)$.
It is sufficient to prove, for each $N\geq 1$,
\begin{equation}\label{tightness-estimate}
  \E \int_0^T \|u_N(t)\|_2^p\, \d t + \E \int_0^T\! \int_0^T \frac{\|u_N(t) - u_N(s)\|_{H^{-\beta}}^4}{|t-s|^{7/3}}\,\d t\d s \leq C,
\end{equation}
where $C$ is independent of $N$. Lemma \ref{energy-estimate} shows $\{ u_N(\cdot)\}_{N\geq 1}$ is bounded in $L^p \big(\Omega,L^p(0,T; H_\sigma) \big)$ for any $p>2$:
\begin{equation}\label{key-bound-1}
  \E \int_0^T \| u_N(t) \|_{L^2}^p \,\d t \leq T\| u_N(0) \|_{L^2}^p \leq T\| u_0 \|_{L^2}^p.
\end{equation}
Similar to the calculation in \cite[Lemma 3.4]{FGL21c}, we can prove that there exists a constant $C>0$ such that for any $N\geq 1$, $0\leq s<t\leq T$ and $k\in \Z^d_0$,
  $$\E \left[ \sum_{i=1}^{d-1} |\< u_N(t) -  u_N(s), \sigma_{k,i}\>|^2 \right]^2 \leq C |k|^8 |t-s|^2 .$$
Taking  $\beta>4 +\frac{d}{2} $ and using Cauchy's inequality, we have
  $$\aligned
  \E \big[\| u_N(t) -  u_N(s)\|_{H^{-\beta}}^4 \big] &= \E \Bigg[\sum_{k,i } \frac{|\< u_N(t) -  u_N(s), \sigma_{k,i}\>|^2 }{|2 \pi k|^{2 \beta}} \Bigg]^2\\
  &\leq \bigg(\sum_{k\in \Z^d_0 } \frac{1 } {|2 \pi k|^{2\beta}} \bigg) \Bigg(\sum_{k\in \Z^d_0 } \frac{\E \big[ \sum_{i=1}^{d-1} |\< u_N(t) -  u_N(s), \sigma_{k,i}\>|^2 \big]^2 } {| 2 \pi k|^{2 \beta}} \Bigg) \\
  &\leq  C|t-s|^2 \sum_{k\in \Z^d_0 } \frac1{|k|^{2(\beta -4)}} \leq C'|t-s|^2.
  \endaligned$$
Then we get
\begin{equation}\label{key-bound-2}
  \E \int_0^T\! \int_0^T \frac{\|u_N(t) - u_N(s)\|_{H^{-\beta}}^4}{|t-s|^{7/3}}\,\d t \, \d s \leq C .
\end{equation}

Thus, we have proved \eqref{tightness-estimate} and obtain the tightness of $\{\eta_N\}_{N\geq 1}$ on $ \mathcal{X}:= C\big( [0,T]; H_{\sigma}^{-} \big)$.
We define the Polish space $\mathcal Y:= C\big([0,T], \R^{\Z^d_0 \times (d-1)} \big)$ endowed with the following metric:
$$d_{\mathcal Y}(w,\hat w) = \sup_{t\in [0,T]} \sum_{k,i} \frac{|w_{k,i}(t)-\tilde{w}_{k,i}(t)| \wedge 1}{2^{|k|}} ,\quad w, \tilde{w} \in \mathcal Y,$$

Denote the whole sequence of processes $\big\{ (W^{k,i}_t)_{0\leq t\leq T}: k \in \Z^d_0, \, i=1,2,\cdots,d-1 \big\} \in \mathcal{Y}$ by $W_\cdot= (W_t)_{0\leq t\leq T}$. For any $N\geq 1$, denote by $P_N$ the joint law of $( u_N(\cdot), W_\cdot )$ on $\mathcal X \times \mathcal Y $.
Since the laws $\{ \eta_N \}_{N\in \N}$ is tight on $\mathcal X$, we conclude that $\{ P_N \}_{N\in \N}$ is tight on $\mathcal X \times \mathcal Y$.
By Prohorov theorem and Skorokhod's representation theorem, there exist a probability space $\big(\tilde\Omega, \tilde{\mathcal F}, \tilde \P \big)$, and a subsequence $\{N_j\}_{j \in \N}$ such that
\begin{equation}\label{strong-convergence}
  \tilde\P \mbox{-a.s.}, \quad \tilde  u_{N_j}(\cdot) \mbox{ converge strongly to } \tilde  u(\cdot) \mbox{ in } C([0,T]; H_{\sigma}^{-}),
\end{equation}
where $\big(\tilde  u_{N_j}(\cdot), \tilde W^{N_j}_\cdot \big)_{i\in \N}$ and $\big(\tilde  u(\cdot), \tilde W_\cdot \big)$ are processes on $\tilde \Omega$ with laws $P_{N_j}$ and $P$, respectively.

Let $\tilde{v}_{N_j}=K(\tilde{u}_{N_j}) $ and $\tilde{v}=K(\tilde{u})$. By the property of operator $K$ we have
\begin{equation}\label{strong-convergence-1}
  \tilde\P \mbox{-a.s.}, \quad \tilde  v_{N_j}(\cdot) \mbox{ converge strongly to } \tilde  v(\cdot) \mbox{ in } C([0,T]; H_{\sigma}^{2\gamma_0-}).
\end{equation}
We are going to prove that $\big(\tilde  u(\cdot), \tilde W_\cdot \big)$ is a weak solution to the equation \eqref{stoch-NS-eq3}. The following Lemma gives the property of $\tilde{u}$, which is similar to \cite[Lemma 3.5]{FGL21c}. We omit the proof.

\begin{lemma}\label{lem-bddness}
The process $\tilde u$ has $\tilde{\mathbb{P}}$-a.s. weakly continuous trajectories in $L^2$ and satisfies
\begin{equation}\label{estim-final}
\sup_{t\in [0,T]} \Vert \tilde u(t)\Vert_{L^2}\leq \Vert  u_0\Vert_{L^2}\quad \tilde{\mathbb{P}}\text{-a.s.}
\end{equation}
\end{lemma}

Finally, we can give the proof of Theorem \ref{thm-existence}.

\begin{proof}[Proof of Theorem \ref{thm-existence}]
The processes $\big(\tilde  u_{N_j}(\cdot), \tilde W^{N_j}_\cdot \big)$ on $\big(\tilde\Omega, \tilde{\mathcal F}, \tilde \P \big)$ have the same laws with that of $( u_{N_j}(\cdot), W_\cdot )$, which satisfy the equation \eqref{SDE} with $N$ replaced by $N_j$. Some classical arguments show that the stochastic integrals involved below make sense, see e.g. \cite[Section 2.6, p.89]{Krylov}. For fixed $L\in \N$, for any $N_j>L$ and $\phi\in H_L$, one has,  $\tilde \P$-a.s for all $t\in [0,T]$,
\begin{equation}\label{proof-0}
  \aligned
  \big\< \tilde u_{N_j}(t),\phi \big\> =&\, \big\<  u_{N_j}(0), \phi \big\> +\int_0^t \big\< \tilde u_{N_j}(s), \tilde v_{N_j}(s)\cdot \nabla \phi \big\>\,\d s + \int_0^t \big\<\tilde u_{N_j}(s), S_{\theta,N_{j}}^{(d)}(\phi) \big\>\,\d s \\
  & -\sqrt{C_d \kappa} \sum_{k,i} \theta_k \int_0^t \big\<\tilde u_{N_j}(s), \sigma_{-k,i}\cdot \nabla \phi \big\>\,\d\tilde W^{N_j,k,i}_s.
  \endaligned
\end{equation}

Except for the Stratonovich-It\^o corrector, the convergence estimates of the other parts are similar to those in \cite[Theorem 2.2]{FGL21c}. For the corrector,
$$\aligned
&\ \E_{\tilde \P} \bigg[\sup_{t\in [0,T]}\bigg| \int_0^t \big\<\tilde u_{N_j}(s), S_{\theta,N_{j}}^{(d)}(\phi) \big\>\,\d s - \int_0^t \big\<\tilde u(s), S_{\theta}^{(d)}(\phi) \big\>\,\d s \bigg| \bigg] \\
\leq &\ \E_{\tilde \P} \bigg[\sup_{t\in [0,T]}\bigg| \int_0^t \big\<\tilde u_{N_j}(s), S_{\theta,N_{j}}^{(d)}(\phi) \big\>\,\d s - \int_0^t \big\<\tilde u_{N_j}(s), S_{\theta}^{(d)}(\phi) \big\>\,\d s \bigg| \bigg]\\
&\, + \E_{\tilde \P} \bigg[\sup_{t\in [0,T]}\bigg| \int_0^t \big\<\tilde u_{N_j}(s), S_{\theta}^{(d)}(\phi) \big\>\,\d s - \int_0^t \big\<\tilde u(s), S_{\theta}^{(d)}(\phi) \big\>\,\d s \bigg| \bigg].
\endaligned$$
It is easy to see $S_{\theta,N_{j}}^{(d)}(\phi) \to S_{\theta}^{(d)}(\phi)$ as $j \to \infty$; and by the $L^2(0,T; H_{\sigma})$ boundedness of ${\tilde{u}_{N_j}}$, the first term on the right-hand side vanishes as $j \to \infty$.
For the second term, using \eqref{strong-convergence} and the bounds in Lemma \ref{energy-estimate}, we know that the quantity in the square bracket tends to 0 $\tilde\P$-a.s.; by Lemma \ref{lem-bddness} and the dominated convergence theorem, the second term tends 0 as $j \rightarrow +\infty$.

Similar to the proof of \cite[Theorem 2.2]{FGL21c}, letting $j \rightarrow +\infty$ in \eqref{proof-0} and for any $\phi \in H_{L}$ we obtain, $\P$-a.s. for all $t\in [0,T]$,
\begin{equation*}
  \<\tilde{u}_t,\phi\> = \<u_0,\phi\> + \int_0^t \<\tilde{u}_s, \tilde{v}_s\cdot \nabla \phi\>\,\d s + \int_0^t \<\tilde{u}_s, S_{\theta}^{(d)}(\phi) \>\,\d s - \sqrt{C_d \kappa} \sum_{k,i} \theta_{k} \int_0^t \<\tilde{u}_s, \sigma_{-k,i} \cdot \nabla \phi\> \,\d \tilde{W}^{k,i}_s.
\end{equation*}
By the arbitrariness of $L \in \N$ and another limit argument, we can prove the above identity also holds for any $\phi \in C^\infty(\T^d;\R^d) \cap H_{\sigma}$.
\end{proof}

\section{Convergence rate of Stratonovich-It\^o correctors} \label{sec-appendix-B}

This section is devoted to the proof of Theorem \ref{thm-Ito-corrector}.
Based on \cite[Section 5]{FL21} and \cite[Appendix]{Luo21b}, we will demonstrate the convergence of the Stratonovich-It\^o correction term for the choice of coefficients in \eqref{theta-N-def}.
First, we give some notations that will be used in this section.

Let $\Pi^\perp$ be the operator which is orthogonal to the Leray projection $\Pi$. If $X$ is a general vector field, then, formally,
\begin{equation}\label{Leray-proj-1}
    \Pi^\perp X = \nabla \Delta^{-1} \div(X).
\end{equation}
On the other hand, if $X= \sum_{l\in \Z^d_0} \sum_{i=1}^{d-1} X_{l,i} e_l$, $X_{l,i} \in \mathbb C^d$, then
\begin{equation}\label{Leray-proj-2}
    \Pi^\perp X= \sum_{l,i} \frac{l\cdot X_{l,i}}{|l|^2} l e_l = \nabla\bigg[ \frac1{2\pi {\rm i}} \sum_{l,i} \frac{l\cdot X_{l,i}}{|l|^2} e_l \bigg].
\end{equation}
By the discussions in \cite[Section 5]{FL21} and \cite[Appendix]{Luo21b}, we have
\begin{equation}\label{decompositions}
    S_\theta^{(d)}(v) = \kappa\Delta v - C_d \kappa \sum_{k,i} \theta_k^2\, \Pi\big[ \sigma_{k,i} \cdot\nabla \Pi^\perp (\sigma_{-k,i} \cdot\nabla v) \big].
\end{equation}
We shall denote the second term on the right-hand side by $S_\theta^{(d),\perp} (v)$.
We assume the divergence free vector field $v \in \mathcal{H}_{\sigma}$ has the Fourier expansion
  $$v= \sum_{l \in \Z^d_0} \sum_{j=1}^{d-1} v_{l,j}\, \sigma_{l,j}, $$
where the coefficients $\{v_{l,j} \} \subset \mathbb C$ satisfy $\overline{v_{l,j}}= v_{-l,j}$.

For any $k,l\in \Z^d$, we define $\angle_{k,l}$ as the angle between $k$ and $l$.
To calculate $S_{\theta}^{(d),\perp}$, we need the following identity proved by \cite[Lemma 6.2]{Luo21b} and \cite[Lemma 5.4]{FL21}.

\begin{lemma}\label{lem-append-exress}
  We have
    $$S_\theta^{(d),\perp} (v)= - 4 \pi^2 C_d \kappa \sum_{l,j} v_{l,j} |l|^2 \sum_{i=1}^{d-1} \bigg[ \sum_k \theta_k^2 \sin^2 (\angle_{k,l}) \frac{(a_{l,i}\cdot (k-l))(a_{l,j}\cdot (k-l))}{|k-l|^2} \bigg] \sigma_{l,i}. $$
\end{lemma}

To prove Theorem \ref{thm-Ito-corrector}, we need to deal with the convergence of the coefficients in the Fourier expansions of $S_{\theta^N}^{(d),\perp} (v)$.
This will be done in several steps, and the first one is to prove the following lemma.

\begin{lemma}\label{lem-differ}
For any $l\in \Z^d_0$ and all $N\geq 1$ it holds
    $$\bigg| \sum_k (\theta^N_k)^2 \sin^2 (\angle_{k,l}) \bigg( \frac{(a_{l,i}\cdot (k-l))(a_{l,j}\cdot (k-l))}{|k-l|^2} - \frac{(a_{l,i} \cdot k)(a_{l,j} \cdot k)  }{|k|^2} \bigg) \bigg| \leq 4|l| D_N. $$
\end{lemma}

\begin{proof}
We have
    $$\bigg| \frac{(a_{l,i}\cdot (k-l))(a_{l,j}\cdot (k-l))}{|k-l|^2} - \frac{(a_{l,i} \cdot k)(a_{l,j} \cdot k)  }{|k|^2} \bigg| \leq \bigg\|\frac{(k-l)\otimes (k-l)}{|k-l|^2} - \frac{k\otimes k}{|k|^2} \bigg\| , $$
where $\|\cdot \|$ is the matrix norm. Note that
    $$\frac{(k-l)\otimes (k-l)}{|k-l|^2} - \frac{k\otimes k}{|k|^2} = \frac{k-l}{|k-l|} \otimes \bigg(\frac{k-l}{|k-l|} - \frac{k}{|k|} \bigg) + \bigg(\frac{k-l}{|k-l|} - \frac{k}{|k|} \bigg) \otimes \frac{k}{|k|},$$
and thus
    $$\bigg\| \frac{(k-l)\otimes (k-l)}{|k-l|^2} - \frac{k\otimes k}{|k|^2} \bigg\| \leq 2\bigg| \frac{k-l}{|k-l|} - \frac{k}{|k|} \bigg|.$$
Next, since
    $$\frac{k-l}{|k-l|} - \frac{k}{|k|}= \bigg(\frac1{|k-l|} - \frac1{|k|}\bigg)(k-l) - \frac{l}{|k|}, $$
one has
    $$\bigg| \frac{k-l}{|k-l|} - \frac{k}{|k|} \bigg| \leq \frac{\big| |k| -|k-l| \big|}{ |k|} + \frac{|l|}{|k|} \leq 2 \frac{|l|}{|k|}. $$
Combining the above estimates we obtain
    $$\bigg| \frac{(a_{l,i}\cdot (k-l))(a_{l,j}\cdot (k-l))}{|k-l|^2} - \frac{(a_{l,i} \cdot k)(a_{l,j} \cdot k)  }{|k|^2} \bigg| \leq 4\frac{|l|}{|k|} . $$
Finally, recalling the definition of $D_N$, it holds
    $$\sum_{k} \big(\theta^N_k \big)^2 \sin^2(\angle_{k,l}) \bigg| \frac{(a_{l,i}\cdot (k-l))(a_{l,j}\cdot (k-l))}{|k-l|^2} - \frac{(a_{l,i} \cdot k)(a_{l,j} \cdot k)}{|k|^2} \bigg| \leq 4 \sum_{k} \big(\theta^N_k \big)^2 \frac{|l|}{|k|} = 4|l| D_N. $$
\end{proof}

To estimate $\sum_k (\theta^N_k)^2 \sin^2 (\angle_{k,l}) \frac{(a_{l,i}\cdot k)(a_{l,j}\cdot k)}{|k|^2}$, it is natural to convert the summation into an integral. The difference between summation and integral is estimated in the following lemma. We write $\square(k)$ for the unit square or cube centered at $k\in \Z^d$ such that all sides have length 1 and are parallel to the coordinate axes.

\begin{lemma}\label{lem-integral}
Let $\{\theta^N \}_{N\geq 1} \subset \ell^2$ be given as in \eqref{theta-N-def}. Assume $g: \R^d \rightarrow \R$ is a bounded function satisfying
\begin{equation*}
    |g(x)-g(k)| \leq \frac{C_1}{|k|} \quad \mbox{for all } x\in \square(k) \mbox{ and } |k|\ge 1.
\end{equation*}
There exists a constant $C>0$, depending on $C_1$ but independent of $N\geq 1$, such that
    \begin{equation*}
        \bigg| \sum_{k} \big(\theta^N_k \big)^2 g(k) - \epsilon_N \int_{\{1 \leq |x|\leq N\}} \frac{g(x)}{|x|^{2\gamma}} \,\d x \bigg| \leq C D_N.
    \end{equation*}
\end{lemma}

\begin{proof}
Without losing generality, we assume $\|g\|_{L^\infty} \leq 1$. Let $\theta^N(x)=\sqrt{\epsilon_N}/|x|^{\gamma}$ and $S_N = \bigcup_{1 \leq |k| \leq N} \square(k)$, then
    $$\aligned
      & \bigg|\sum_{k} \big(\theta^N_k \big)^2 g(k) - \int_{S_N} \frac{\epsilon_N}{|x|^{2\gamma}} g(x) \,\d x\bigg| \leq \sum_{1 \leq |k| \leq N} \int_{\square{(k)}} \bigg| (\theta^N(k))^2 g(k) - \frac{\epsilon_N}{|x|^{2\gamma}} g(x) \bigg| \,\d x  \\
       \quad & \leq \sum_{1 \leq |k| \leq N} \int_{\square{(k)}} (\theta^N(k))^2 \big| g(k) - g(x) \big| \,\d x + \sum_{1 \leq |k| \leq N} \int_{\square{(k)}} \bigg| (\theta^N(k))^2-\frac{\epsilon_N}{|x|^{2\gamma}} \bigg| \big| g(x) \big| \,\d x.
      \endaligned
    $$
For all $|k| \geq 1$ and $x\in \square(k)$, we have $|x-k|\leq \sqrt{d}/2$ and $|x|\geq 1/2$, thus
\begin{equation*}
    \bigg| (\theta^N(k))^2-\frac{\epsilon_N}{|x|^{2\gamma}} \bigg| = \epsilon_N \biggl| \frac{1}{|k|^{2\gamma}}-\frac{1}{|x|^{2\gamma}} \biggr| \leq C_2 \frac{\epsilon_N}{|k|^{2\gamma+1}}= C_2 \frac{(\theta^N(k))^2}{|k|}
\end{equation*}
and $|g(x)-g(k)| \leq \frac{C_1}{|k|}$. Then we have
\begin{equation*}
    \bigg|\sum_{k} \big(\theta^N_k \big)^2 g(k) - \epsilon_N  \int_{S_N} \frac{g(x)}{|x|^{2\gamma}} \,\d x\bigg| \leq C_3 \sum_{1 \leq |k| \leq N} \int_{\square{(k)}} \frac{(\theta^N(k))^2}{|k|}  \,\d x = C_3 D_N.
\end{equation*}

Note that there is a small difference between the sets $\{1 \leq |x|\leq N\}$ and $S_N$, but, in the same way, one can show that
    $$\bigg|\int_{\{1 \leq |x|\leq N\}} \frac{g(x)}{|x|^{2\gamma}} \,\d x - \int_{S_N} \frac{g(x)}{|x|^{2\gamma}} \,\d x\bigg| \leq C \frac{D_N}{\epsilon_N} .$$
Indeed, for any $x\in \square(k)$ with $1 \leq |k| \leq N$, one has $\frac{1}{2} \leq |x| \leq N+1$. Therefore,
    $$S_N = \bigcup_{1 \leq |k| \leq N} \square(k) \subset \Big\{ x \in \R^d: \, \frac{1}{2} \leq |x| \leq N+1 \Big\} =: T_N. $$
One also has
    $$R_N:= \{1 \leq |x| \leq N-1 \} \subset S_N.$$
Let $A\Delta B$ be the symmetric difference of subsets $A,B\subset \R^d$; then,
    $$\aligned
      &\bigg|\int_{\{1 \leq |x|\leq N\}} \frac{g(x)}{|x|^{2\gamma}} \,\d x - \int_{S_N} \frac{g(x)}{|x|^{2\gamma}} \,\d x\bigg| \leq \int_{S_N\Delta \{1 \leq |x|\leq N\}} \frac{|g(x)|}{|x|^{2\gamma}} \,\d x  \\
      & \, \leq \int_{S_N\Delta \{1 \leq |x|\leq N\}} \frac{1}{|x|^{2\gamma}} \,\d x \leq \int_{T_N\setminus R_N} \frac{1}{|x|^{2\gamma}} \,\d x
      \leq C_4 \Big(\frac{1}{N^{2\gamma-d+1}} + 1 \Big) \leq C_5 \frac{D_N}{\epsilon_N},
    \endaligned $$
where the last step follows from
    $$\frac{D_N}{\epsilon_N} = \sum_{1 \leq |k| \leq N} \frac1{|k|^{2\gamma+1}} \geq 1+ \frac1{N^{2\gamma+1}}\, \#\{k\in \Z^d_0: 2 \leq |k| \leq N \} \geq 1+\frac{C_6}{N^{2\gamma-d+1}}.$$
It is easy to see that the constants $C_4, C_5$ and $C_6$ depend only on $\gamma$.
\end{proof}

\begin{corollary} \label{corollary-intergral}
Let $\{\theta^N \}_{N\geq 1} \subset \ell^2$ be given as in \eqref{theta-N-def}. There exists a $C>0$, independent of $N\geq 1$ and $l\in \Z^d_0$, such that
  \begin{equation*}
  \bigg| \sum_{k} \big(\theta^N_k \big)^2 \sin^2(\angle_{k,l}) \frac{(a_{l,i}\cdot k)(a_{l,j}\cdot k)}{|k|^2} - \epsilon_N \int_{\{1 \leq |x|\leq N\}} \frac{\sin^2(\angle_{x,l})}{|x|^{2\gamma}} \frac{(a_{l,i}\cdot x)(a_{l,j}\cdot x)}{|x|^2} \,\d x \bigg| \leq C D_N.
  \end{equation*}
\end{corollary}

\begin{proof}
For any $l \in \Z_0^d$ and index $i,j \in \{1,\cdots,d-1\}$, we define the function
    $$g_{l,i,j}(x) = \sin^2(\angle_{x,l}) \frac{(a_{l,i} \cdot x)(a_{l,j} \cdot x)}{|x|^2}, \quad x\in \R^d,\, x\neq 0; $$
clearly, $\|g_{l,i,j} \|_{L^\infty} \leq 1$.
For any $|k|\geq 1$ and $x\in \square(k)$, we have
  $$\aligned
  |g_{l,i,j}(k) -g_{l,i,j}(x)| &\leq |\sin^2(\angle_{k,l}) -\sin^2(\angle_{x,l})| + \bigg| \frac{(a_{l,i}\cdot k)(a_{l,j}\cdot k)}{|k|^2} - \frac{(a_{l,i}\cdot x)(a_{l,j}\cdot x)}{|x|^2} \bigg| \\
  &\leq 2|\sin(\angle_{k,l}) -\sin(\angle_{x,l})| + \bigg\| \frac{k\otimes k}{|k|^2} - \frac{x\otimes x}{|x|^2} \bigg\| \\
  &\leq 2| \angle_{k,l} -\angle_{x,l}| + 2 \bigg| \frac{k}{|k|} - \frac{x}{|x|}\bigg|.
  \endaligned $$
Since $|x-k| \leq \sqrt{d}/2$ and $|k| \geq 1$, we can find a constant $C_1>0$, independent of $l\in \Z^d_0$, indices $i,j \in \{1,\cdots,d-1\}$ and $N\geq 1$, such that
  $$|g_{l,i,j}(k) -g_{l,i,j}(x)| \leq \frac{C_1}{|k|}. $$
Then by Lemma \ref{lem-integral}, we finish the proof.
\end{proof}

Now we are ready to provide the proof of Theorem \ref{thm-Ito-corrector}.

\begin{proof}[Proof of Theorem \ref{thm-Ito-corrector}]
  We will first give detailed proof for the 3D case and then briefly explain the difference between the 2D and 3D proofs latter.

  For $d=3$, by Lemma \ref{lem-append-exress}, we have
    $$ S_\theta^{(3),\perp} (v)= - 6 \pi^2 \kappa \sum_{l,j} v_{l,j} |l|^2 \sum_{i=1}^{2} \bigg[ \sum_k \theta_k^2 \sin^2 (\angle_{k,l}) \frac{(a_{l,i}\cdot (k-l))(a_{l,j}\cdot (k-l))}{|k-l|^2} \bigg] \sigma_{l,i}. $$
  Since $v=\sum_{l,j} v_{l,j} \sigma_{l,j}$ one has
    $$ \kappa \Delta v= -4 \pi^2\kappa \sum_{l,j} v_{l,j} |l|^2 \sigma_{l,j}. $$
  Then we have
    $$\aligned
    & S_{\theta^N}^{(3),\perp} (v) -\frac{2}{5} \kappa \Delta v \\
    & \quad = - 6 \pi^2 \kappa \sum_{l} \sum_{i,j=1}^{2} v_{l,j} |l|^2 \bigg[ \sum_k (\theta_k^N)^2 \sin^2 (\angle_{k,l}) \frac{(a_{l,i}\cdot (k-l))(a_{l,j}\cdot (k-l))}{|k-l|^2} -\frac{4}{15} \delta_{i,j} \bigg] \sigma_{l,i},
    \endaligned $$
  Then fix any $L>0$, we have
    $$\bigg\| S_{\theta^N}^{(3),\perp}(v) -\frac{2}{5} \kappa \Delta v \bigg\|_{H^{b-2-\alpha}}^2 \leq 2 K_{L,1} + 2 K_{L,2},$$
  where ($C=36 \pi^4$)
    $$\aligned
    K_{L,1} &= C\kappa^2 \sum_{|l|> L} \sum_{i,j=1}^{2} |v_{l,j}|^2 |l|^{2(b-\alpha)} \bigg| \sum_k (\theta_k^N)^2 \sin^2 (\angle_{k,l}) \frac{(a_{l,i}\cdot (k-l))(a_{l,j}\cdot (k-l))}{|k-l|^2} -\frac{4}{15} \delta_{i,j} \bigg|^2 ,\\
    K_{L,2} &= C\kappa^2 \sum_{|l|\leq L} \sum_{i,j=1}^{2} |v_{l,j}|^2 |l|^{2(b-\alpha)} \bigg| \sum_k (\theta_k^N)^2 \sin^2 (\angle_{k,l}) \frac{(a_{l,i}\cdot (k-l))(a_{l,j}\cdot (k-l))}{|k-l|^2} -\frac{4}{15} \delta_{i,j} \bigg|^2.
    \endaligned$$
  Since $|a_{l,i}|=|a_{l,j}|=1$ and $\|\theta^N \|_{\ell^2}=1$, one has
    $$\aligned
    K_{L,1} &\leq C\kappa^2 \sum_{|l|> L} \sum_{j=1}^{2}  |v_{l,j}|^2|l|^{2(b-\alpha)} \bigg(\sum_{k} \big(\theta^N_k \big)^2  + \frac{4}{15} \bigg)^2\\
    &\leq C' \kappa^2 \sum_{|l|> L} \sum_{j=1}^{2}  |v_{l,j}|^{2} \frac{|l|^{2b}}{L^{2\alpha}} \leq C' \kappa^2 \frac{1}{L^{2\alpha}} \|v \|_{H^b}^2.
    \endaligned $$
  Next we estimate the second term $K_{L,2}$. We claim that there exists some constant $C>0$ such that for any $l\in \Z^3_0$, $i,j \in \{1,2\}$ and for all $N \geq 1$, it holds
  \begin{equation} \label{calculation-3d}
    \bigg| \sum_{k} \big(\theta^N_k \big)^2 \sin^2(\angle_{k,l}) \frac{(a_{l,i} \cdot (k-l))(a_{l,j} \cdot (k-l))}{|k-l|^2}- \frac{4}{15} \delta_{i,j} \bigg| \leq C D_N |l| .
  \end{equation}
  By this assertion, for $\alpha \in [0,1]$ we have
  $$K_{L,2} \leq C' \kappa^2 L^{2(1-\alpha)} D_N^2 \sum_{|l|\leq L} \sum_{j=1}^{2} |l|^{2b} |v_{l,j}|^2 \leq C' \kappa^2 D_N^2 L^{2(1-\alpha)} \|v \|_{H^b}^2.$$
  Summarizing these estimates and taking $L=D_N^{-1}$, we have
  $$ \bigg\| S_{\theta^N}^{(3),\perp}(v) -\frac{2}{5} \kappa \Delta v \bigg\|_{H^{b-2-\alpha}} \leq C \kappa D_N^{\alpha} \|v \|_{H^b}.$$
  By the identity \eqref{decompositions}, we complete the proof of Theorem \ref{thm-Ito-corrector} in 3D-dimension.

  Now, we turn to prove the assertion \eqref{calculation-3d}. Let
    $$
    J^{(i,j)}_N :=\epsilon_N \int_{\{1 \leq |x|\leq N\}} \frac{\sin^2(\angle_{x,l})}{|x|^{2\gamma}} \frac{(a_{l,i}\cdot x)(a_{l,j}\cdot x)}{|x|^2} \,\d x.
    $$
  By Lemma \ref{lem-differ} and Corollary \ref{corollary-intergral}, we only need to prove $|J^{(i,j)}_N-\frac{4}{15}\delta_{i,j}| \leq C D_N $. By the calculation in \cite[Section 5]{FL21}, $J_N^{(i,j)}$ is independent of $l \in \Z_0^3$ and
  \begin{equation} \label{proof-J_N-3}
    J_N^{(i,j)} = \frac{16 \pi}{15} \epsilon_N \int_1^{N}\frac{\d r}{r^{2\gamma-2}} \delta_{i,j}.
  \end{equation}
And by Lemma \ref{lem-integral} with $g\equiv \frac{4}{15}$, we have
\begin{equation*}
  \bigg| \frac{4}{15} \epsilon_N \int_{\{1 \leq |x|\leq N\}} \frac{\d x}{|x|^{2\gamma}} -\frac{4}{15} \sum_{k} \big(\theta^N_k \big)^2 \bigg|\leq C D_{N}
\end{equation*}
for some constant $C>0$. Equivalently,
  $$ \bigg| \frac{16\pi}{15} \epsilon_N \int_1^{N} \frac{\d r}{r^{2\gamma -2}} -\frac{4}{15} \bigg|\leq C D_N.$$
Recalling \eqref{proof-J_N-3}, we obtain the assertion \eqref{calculation-3d}.

For $d=2$, we have
$$S_{\theta^N}^{(2),\perp} (v) -\frac{3}{4} \kappa \Delta v = - 8 \pi^2 \kappa \sum_{l} v_{l} |l|^2 \bigg[ \sum_k \theta_k^2 \sin^2 (\angle_{k,l}) \frac{(a_{l}\cdot (k-l))^2}{|k-l|^2} -\frac{3}{8} \bigg] \sigma_{l}.$$
By the calculation in \cite[Appendix]{Luo21b}, we have
$$ J_N := \epsilon_N \int_{\{1 \leq |x|\leq N\}} \frac1{|x|^{2\gamma}} \sin^2(\angle_{x,l})  \frac{(a_l\cdot x)^2}{|x|^2} \,\d x=\frac{3\pi}{4} \epsilon_N \int_{1}^{N} \frac{\d r}{r^{2\gamma-1}} .$$
Then, in the same way, we can prove $\big|J_N-\frac{3}{8}\big|<CD_N$. By Lemma \ref{lem-differ} and Corollary \ref{corollary-intergral}, there exists some constant $C>0$ such that for any $l\in \Z^2_0$ and all $N \geq 1$, it holds
$$ \bigg| \sum_{k} \big(\theta^N_k \big)^2 \sin^2(\angle_{k,l}) \frac{(a_l\cdot (k-l))^2}{|k-l|^2}- \frac38 \bigg| \leq C D_N |l| .$$
Similar as in the three-dimensional case, we can also prove
$$ \bigg\| S_{\theta^N}^{(2),\perp} (v) -\frac{3}{4} \kappa \Delta v \bigg\|_{H^{b-2-\alpha}} \leq C \kappa D_N^{\alpha} \|v \|_{H^b}.$$
By the identity \eqref{decompositions}, we finish the proof of Theorem \ref{thm-Ito-corrector} in the 2D case.
\end{proof}

\bigskip

\noindent \textbf{Acknowledgements:} Both authors are very grateful to Professor Zhao Dong for helpful discussions. They would like to thank the financial supports of the National Key R\&D Program of China (No. 2020YFA0712700), the National Natural Science Foundation of China (Nos. 11931004, 12090014). The first named author is also supported by the Youth Innovation Promotion Association, CAS (Y2021002). \smallskip

\noindent \textbf{Conflicts of Interest:} The authors declare no conflicts of interest. \smallskip

\noindent \textbf{Date Availability:} The paper has no data to share.

\end{document}